\documentclass[twoside,twocolumn,10pt]{article}
\usepackage[sc]{mathpazo}
\usepackage[T1]{fontenc} 
\usepackage{microtype} 
\usepackage{bm}
\usepackage{amsmath,amssymb,amsthm}
\usepackage[english]{babel} 

\usepackage[hmarginratio=1:1,right=1.5cm,top=2cm,bottom=2cm,left=1.5cm,columnsep=20pt]{geometry} 
\usepackage[hang, small,labelfont=bf,up,textfont=it,up]{caption} 
\usepackage{booktabs,pgfplots,tikz} 
\pgfplotsset{compat=1.13}
\usepackage{algorithm}
\usepackage[noend]{algpseudocode}
\usepackage{booktabs}
\usepackage{filecontents}
\usepackage[export]{adjustbox}

\usepackage[colorlinks=true,linkcolor=magenta,citecolor=cyan]{hyperref}

\newcommand{\norm}[1]{\left\lVert#1\right\rVert}
\DeclareMathOperator*{\argmin}{arg\,min}
\DeclareMathOperator*{\argmax}{arg\,max}
\newtheorem{remark}{Remark}
\newtheorem{proposition}{Proposition}

\usetikzlibrary{patterns}

\allowdisplaybreaks
\newtheorem{theorem}{Theorem}

\algnewcommand\algorithmicinput{\textbf{Input:}}
\algnewcommand\Input{\item[\algorithmicinput]}

\algnewcommand\algorithmicoutput{\textbf{Output:}}
\algnewcommand\Output{\item[\algorithmicoutput]}

\usepackage{enumitem} 
\setlist[itemize]{noitemsep} 

\usepackage{abstract} 

\usepackage{titlesec} 
\titleformat{\section}[block]{\large\scshape\centering}{\thesection.}{1em}{} 
\titleformat{\subsection}[block]{\large}{\thesubsection.}{1em}{} 
\providecommand{\keywords}[1]{\textbf{\textit{Keywords---}} #1}
\usepackage{fancyhdr} 
\usepackage{titling} 

\pretitle{\begin{center}\LARGE\bfseries} 
\posttitle{\end{center}} 
\title{Estimation of Location and Orientation for \\Underwater Vehicles from Range Measurements} 
\usepackage{authblk}
\author[$\dagger$]{Sai Krishna Kanth Hari}
\author[$\ddagger$]{Kaarthik Sundar}
\author[1]{Jose Braga}
\author[1]{\\Joao Teixeira}
\author[1]{Joao Sousa}
\author[$\dagger$]{Swaroop Darbha\thanks{Corresponding author: dswaroop@tamu.edu}}
\affil[$\dagger$]{\small Texas A\&M University, College Station, Texas, USA}
\affil[$\ddagger$]{\small Center for Nonlinear Studies, Los Alamos National Laboratory, USA}
\affil[1]{\small LSTS, Underwater Systems and Technology Laboratory, University of Porto, Portugal}

\date{} 
\renewcommand{\maketitlehookd}{%
\begin{abstract}
\noindent  Localization is an important required task for enabling vehicle autonomy for underwater vehicles. Localization entails the determination of position of the center of mass and orientation of a vehicle from the available  measurements. In this paper, we focus on localization by using One-Way Travel Time (OWTT) measurements available to a vehicle from the communication of its multiple on-board receivers with acoustic beacons, more specifically, long baseline (LBL) beacons. Range can be inferred by multiplying the OWTT with speed of sound; however, water conditions can change spatially and temporally resulting in uncertainty in range measurement. The farther a beacon is from a receiver, the larger is the uncertainty. The proposed method for localization accounts captures this uncertainty by bounding the true distance with an increasing (calibrating) function of the range measurement. Determination of this calibration function is formulated as polynomial optimization problem and is a crucial step for localization. The proposed two-step procedure for localization is as follows: based on the range measurements specific to a receiver from the beacons, a convex optimization problem is proposed to estimate the location of the receiver. The estimate is essentially a center of the set of possible locations of the receiver. In the second step, the location estimates of the vehicle are corrected using rigid body motion constraints and the orientation of the rigid body is thus determined. Numerical examples and experimental results provided at the end corroborate the procedures developed in this paper.\\

\noindent\keywords{localization, underwater vehicles, range measurements, convex optimization, polynomial optimization}
\end{abstract}
}

\begin{document}
\maketitle

\section{Introduction}\label{sec:intro}

Over the past decade, autonomous underwater vehicles (AUVs) have gathered tremendous significance due to their ability to operate in extreme environmental conditions and perform complex commercial, civilian, and military tasks that were previously considered unrealistic. A few such tasks include oil and gas exploration, performing mine countermeasures, finding wreckage of missing airplanes, studying the physical and biological properties of the oceans, etc. In all the aforementioned missions, the problem of localization of underwater vehicles, which deals with the estimation of location and orientation of vehicles, is important for autonomous navigation, guidance, and control \cite{Leonard1998}. Knowledge of location and orientation is necessary to ensure that a vehicle autonomously tracks its desired trajectory for vehicle control, safe operation, and recovery. Localization therefore requires sensing, and correspondingly, localization procedures are dependent on the available sensory measurements. Current underwater vehicle platforms are equipped with IMU/GPS, compass, Doppler Velocity Logs (DVL), sonar, acoustic transponders or beacons, video-based sensors, pressure sensors, communication devices or some subset of them to aid the localization \cite{Paull2014}. Since sensory measurements are usually contaminated with noise, the problem of localization also requires filtering the noise in order to determine an accurate estimate of location and orientation. Numerous algorithms have been developed for localization of aerial and ground vehicles \cite{Sundar2017,Sundar2018}. However, localization in underwater environments is quite challenging due to its harsh ambient conditions \cite{Zhang2015}.

The physical characteristics of an underwater environment make it practically impossible to use radio waves for long distance communication; this is due to the rapid attenuation of high frequency electromagnetic waves underwater. Hence, the traditional global positioning system (GPS) cannot be used for underwater localization. In the past, techniques that utilize external environmental features such as magnetic field maps \cite{Armstrong2010, Armstrong2009}, gravitational anomaly \cite{Tang2017, Bishop2002} have been employed. These methods fall under the category of simultaneous localization and mapping (SLAM), where geophysical parameters are compared to existing field maps to determine the location of vehicles. Optical imagery is another well-used SLAM method. However, such methods face a number of practical difficulties due to the unstructured nature of the underwater environment. Line-of-sight based communication using optical signals, also faces the problem of low transmission distances. In fact, acoustic signals are known to outperform their counterparts such as radio and optical signals \cite{Qureshi2016}. Hence, using acoustic sensor networks has become the most convenient mode of underwater communication.  To that end, Underwater Acoustic Sensor Networks (UASNs) are widely used for underwater communications. UASNs involve deployment of a variable number of sensors and vehicles to perform collaborative monitoring tasks over a given area. Localization schemes using UASNs can be broadly classified into two types: range-free schemes and range-based schemes. Range-free schemes, as the name suggests, do not use any range measurements for localization; examples of range-free schemes include single hop and multihop count-based algorithms \cite{Niculescu2003, Wong2005}, area and centroid based algorithms \cite{Bulusu2000}, to name a few. On the other hand, range-based methods utilize bearing or range information obtained from Time of Arrival (TOA) or Time Difference of Arrival (TDOA) of acoustic signals to estimate their distances from desired nodes in the network. Range-based schemes can in turn be classified into three main categories: infrastructure-based methods, distributed positioning schemes, and schemes that use mobile beacons. In general, the localization estimates provided by range-based methods are more precise compared to those provided by the range-free methods \cite{Chandrasekhar2006}.  However, acoustic communication faces a myriad of challenges underwater such as data losses, high latency, high volatility of sound speed (due to continuous changes in water temperature and salinity) etc., failing to provide accurate range measurements. 

In this paper, we address the problem of infrastructure-based localization in which a set of acoustic sensors equipped with transmitters are deployed at predetermined locations. Common acoustic sensors include short baseline (SBL), ultrashort baseline (USBL) and long baseline (LBL) beacons. As the name suggests, this classification is made on the basis of the area of coverage of these sensors. The range of USBL and SBL beacons is short, and they usually require a surface ship from which they can be suspended. While the USBL beacons are limited by their range, the accuracy of SBL beacons depends on the size of the baseline. In the current work, we use LBL beacons due to their vast area of coverage. These beacons can be deployed at the desired locations by mooring them to sea floor. Additionally, surface beacons such as GPS Intelligent Buoys (GIBs) can also be used \cite{Paull2014}.  Multiple receivers on-board the AUV, whose position and orientation is to be estimated, receive the transmitted signal from the beacons. Either the two-way travel time (TWTT) or the one-way travel time (OWTT) with time synchronization \cite{Eustice2011} can be used to obtain the distance between the beacons and the receivers on-board the vehicle. In the ideal scenario of noise-free sensory measurements, one can use triangulation or multilateration to determine the location and orientation of the AUV. Typically, a significant error in the range measurements is caused by an error in the sound speed profile. Even if the sound speed profile is known at the start of the mission, the acoustic propagation environment can change during the mission. This issue has been dealt with by accounting for the sound speed variations in space and time, and in the process provide more accurate range measurements \cite{Deffenbaugh1994, Deffenbaugh1996}. Techniques such as signal processing and statistical filtering have been used to improve the accuracy of range measurements \cite{Ash2005}. 

The noise in the range measurements causes uncertainty in the on-board receivers' locations which in turn results in an uncertainty in the AUV's location. Existing technologies such as Seaweb technology \cite{Hahn2005} (one of the schemes implemented by the U.S. Navy), use naive least-squares approaches to estimate the location of the receivers and hence, the vehicle. Nevertheless, the solutions obtained by these methods fail to take advantage of well established error patterns. For instance, the accuracy of the range measurements is known to reduce with increase in distance between sensor-receiver pairs. Moreover, in \cite{Garcia2005}, the author highlights the importance of embedding the variations in physical properties of water such as temperature, salinity etc., into localization schemes. Hence, there is a need for an appropriate position estimation algorithm that heeds such features. With regards to the orientation estimation, existing technologies include magnetic sensors, roll and pitch estimation sensors such as pendulum-tilt, fluid-level sensors, and gyro-stabilized versions of these sensors. One of the major drawbacks of magnetic sensors include the magnetic disturbance of the vehicle itself. Gravity-based roll-pitch compensation techniques employed to address this issue face the problem of degraded accuracy in the presence of heave, surge and sway (induced acceleration) of the vehicle \cite{Whitcomb1999,Kinsey2004}. These sensors also face the problem magnetic anomalies near hydro-thermal vents on mid-ocean ridges. The accuracy of pendulum and fluid level sensors also degrade in the presence of induced acceleration. High accurate navigation systems employ true north seeking gyro-compasses to estimate the orientation of the vehicles \cite{Kinsey2006}. This article proposes a two-step, inexpensive technique for localization that takes into account all the issues presented in this paragraph. 

In this work, we primarily develop an optimization-based localization scheme, with a robust error model equipped to account for highly variable physical properties in an underwater setting\footnote{ Preliminary versions of this article without detailed theoretical and experimental results have been published as conference papers \cite{Hari2015, Hari2017}.}. An important assumption we make is that the AUV always travels in a region defined by the convex hull of the beacon locations. The model proposed for measurement assumes that the true distance between a receiver and a beacon is at most equal to a predetermined calibration/sensing function of the range measurement. Determination of this function is formulated as polynomial optimization problem and is a crucial step for localization. The proposed two-step procedure for localization is as follows: based on the range measurements specific to a receiver from the beacons, a convex optimization problem is proposed to estimate the location of the receiver. The estimate is essentially a center of the set of possible locations of the receiver. In the second step, the location estimates of the vehicle are corrected using rigid body motion constraints and the orientation of the rigid body is thus determined. Numerical examples and experimental results provided at the end corroborate the procedures developed in this paper. The proposed technique also provides a scope for online re-calibration, when the vehicle maneuvers outside the range where the calibration/sensing function is not computed a priori. Furthermore, unlike well-known techniques like dead reckoning, our approach provides bounded error estimates i.e., errors do not propagate with time. Our method also provides a measure of uncertainty of the location estimates. Another important problem in infrastructure-based localization is the continuous movement of beacons due to water currents. The problem setup and formulation, as presented in the forthcoming sections, generalizes to handle uncertainty in the beacon locations. 

None of the existing localization methods provide a perfect solution to all the challenges faced in an underwater environment \cite{Kinsey2006}. Furthermore, the applicability of most of the algorithms is restricted to a single ocean depth zone (surface, mid-zone or near bottom zone). The proposed algorithm, even though currently restricted to the convex hull of the beacons, is applicable to all depth zones. Hence, based on the application, the current algorithm can either be operated independently as an inexpensive primary scheme, or can be used as a secondary scheme to aid the existing methods. The rest of the paper is organized as follows: In \ref{sec:setup} and \ref{sec:formulation}, we describe the setup of physical problem and its mathematical formulation, respectively. In \ref{sec:algorithms}, we present formulations and algorithms to solve the associated optimization problems for the location and orientation estimation and in \ref{sec:results}, we corroborate the effectiveness of our proposed approach through extensive numerical and experimental results. Concluding remarks are provided in \ref{sec:conclusion}.

\section{Problem setup} \label{sec:setup}
Consider an underwater sensor network, consisting of $N$ acoustic LBL beacons positioned at previously determined locations, to aid an AUV in its localization. Let $I := \{1,2,\dots,N\}$ denote the set of $N$ beacons. Let the coordinates of these beacons in a global reference frame be $\bm{B}_i := (x^b_i,y^b_i,z^b_i)$ for every $i \in I$. In practice, there is uncertainty associated with these beacon locations and the algorithms that are presented in this article extend to that case as well. For ease of exposition, we have assumed that there is no uncertainty associated with these beacon locations and we do not present the formulations and associated algorithms with the uncertainty in beacon locations.  
The AUV, whose position and orientation are to be determined, is equipped with $L$ receivers on-board denoted by the set $J := \{1,2,\dots,L\}$. The AUV assumed to navigate within the convex hull of these beacons. Let the location estimates of these receivers in the global reference frame be $\bm{r}_j := (x_j,y_j,z_j)$, $j \in J$. Each beacon $i\in I$ is assumed to communicate with each receivers on-board, $j \in J$, to estimate the ranges between beacon $i$ and receiver $j$ utilizing the OWTT with synchronized clocks. A cut-off time is selected for the acoustic signals to reach from the beacons to the receivers, and all the range measurements obtained before the cut-off time are utilized for localization. Let $D_{ij}$ be the measured distance of the on-board receiver $j$ in the vehicle from the beacon $i$ while $d_{ij}$ is the true value. These measurements need not correspond to its true value; let $\phi(D_{ij})$ be an increasing function that provides an upper bound on the true distance, $d_{ij}$, \emph{i.e.}, $d_{ij} \leqslant \phi(D_{ij})$. Essentially, the sensing model, as defined by the calibration function $\phi(\cdot)$, indicates that the true position of the receiver, $j$, will always be contained in the sphere of radius $\phi(D_{ij})$ and centered at the beacon, $i$. Additionally, the calibration function, when modeled in the aforementioned manner, would account for various errors arising in an underwater setting due to variation of the speed of sound, small change in vehicle's location before range measurements from various beacons are obtained, changes in the physical properties of water etc. The determination of this function is obtained using a set of range measurements obtained from the on-board receivers. The problem of determining this calibration function $\phi(\cdot)$ is formulated as a semi-definite program in section \ref{subsec:calibration}. This sensing model has been chosen so that it is amenable to experimental corroboration and makes the subsequent formulation cleaner. Using the aforementioned notations, the problem of localization is stated as follows:\\

\noindent \emph{Determine the estimates $(x_j,y_j,z_j)$, $j\in J$ of the location of the $L$ on-board receivers as well as the orientation of the vehicle that is treated as a rigid body.\\}

We remark that even though the standard pressure sensors such as strain gauges and quartz crystals provide accurate measurements of water depth, they need not always correspond to the geodetic altitude of the vehicle. This is due to the variation of ocean's surface \cite{Kinsey2006}. Hence, the problem is set up in 3-D with the view that 3-D estimates provided by the proposed algorithm would be useful. However, in applications where the depth measurements from these standard sensors sufficiently accurate, they can be combined together with the estimates obtained by setting up a 2-D version of the problem to achieve vehicle localization.

\section{Mathematical formulation} \label{sec:formulation}
In this section, we present a mathematical formulation of the following three problems: (i) optimal estimation of the location of the on-board receiver $j \in J$, (ii) determination of the calibration function $\phi(\cdot)$ for the sensing model, and (iii) estimation of orientation of the vehicle from the location estimates of the on-board receivers. 

\subsection{Optimal position estimation of the $j$\textsuperscript{th} on-board receiver} \label{subsec:formulation-position} 
Let $j \in J$ denote any arbitrary on-board receiver attached to the AUV. Using the distance measurements from each beacon $i\in I$ to the receiver $j$, it is readily clear that 
\begin{flalign}
\norm{\bm r_j - \bm B_i} = d_{ij} \leqslant \phi(D_{ij}), ~ i\in I. \label{eq:sphere}
\end{flalign}
The above set of $N$ distance constraints are convex in $\bm r_j$; note that $\bm B_i$ are known a priori. Let $\mathcal F_j$ be the feasible values of $\bm r_j$ for the above set of constraints. Essentially, the feasible set is the set obtained by intersecting spheres centered at the beacons and of radii determined by the range measurements gathered by the receiver $j \in J$; hence, $\mathcal F_j$ is a convex set. The feasible set indicates the set of all possible locations of the receiver, $j$. The center of the set $\mathcal F_j$ can be considered as  the best estimate of $\bm r_j$. We consider two notions of ``center'' of the set $\mathcal F_j$: the center of the largest inscribed disk (referred to as Chebychev center) and the center of the maximum volume inscribed ellipsoid. The former can be computed via linear programming while the latter, via semi-definite programming. 
\subsubsection{Chebychev center of $\mathcal F_j$} \label{subsubsec:cc} 
If $\bm u$ is any unit vector and $\bm r_{c,j}$ is the center of the largest inscribed disk in $\mathcal F_j$, then the Chebychev center of $\mathcal F_j$ can be computed by the following optimization problem:
\begin{flalign}
 (\mathcal L_1) \quad &\max_{l, \bm r_{c,j}} ~ l, \text{~subject to:}  \label{eq:cc-obj} \\
 \norm{\bm r_{c,j} + l \bm u - \bm B_i} &\leqslant \phi(D_{ij}), ~ i\in I, ~ \forall \bm u.  \label{eq:cc-spheres}
\end{flalign}
The constraints \eqref{eq:cc-spheres} are convex conic constraints \cite{Boyd2004} and describe a disk. The radius of the disk $l$ is defined by the model for measurement and is characterized by the function $\phi(\cdot)$. These constraints can be relaxed to linear constraints in the same manner as approximating a disk by a regular polygon circumscribing the disk. Note that for any vector $\bm x$, its norm $\norm{\bm x} = \max_{\bm v:\norm{\bm v}=1}~\bm v \cdot \bm x$. Hence, the conic constraints in Eq. \eqref{eq:cc-spheres} can be recast as a semi-infinite set of linear inequalities \cite{Boyd2004} as:
\begin{flalign*}
 \max_{\bm v:\norm{\bm v}=1}\left(\bm v \cdot \bm r_{c,j} + l \bm v \cdot \bm u - \bm v \cdot \bm B_i\right) \leqslant \phi(D_{ij}), ~i\in I, \forall \bm u. 
\end{flalign*}
The tightest inequality in the above constraint corresponds to $\bm u = \bm v$ and hence, the semi-infinite linear program describing the location estimation problem is:
\begin{flalign*}
& (\mathcal L_2) \quad \max_{l, \bm r_{c,j}} ~ l, \text{~subject to:  } \\
& \bm v \cdot \bm r_{c,j} + l  \leqslant  \bm v \cdot \bm B_i + \phi(D_{ij}), ~i \in I, \forall \bm v: \norm{\bm v}=1.
\end{flalign*}
One known issue with the Chebychev center is that it can be non-unique depending on the feasible region $\mathcal F_j$ \cite{Boyd2004}. This is not the case with the maximum volume inscribed ellipsoid, presened in the next section.

\subsubsection{Maximum volume inscribed ellipsoid of $\mathcal F_j$} \label{subsubsec:mve} 
The volume of the feasible set $\mathcal F_j$ is a measure of the uncertainty/confidence associated with the position of the $j$\textsuperscript{th} on-board receiver. Here, the ``best'' estimate is given by the center of the maximum volume ellipsoid $\mathcal E_j$ contained in $\mathcal F_j$. The rationale behind this approach is that the volume $V_j$ of this ellipsoid $\mathcal E_j$ provides a lower bound for the volume of $\mathcal F_j$ while the minimum volume of an ellipsoid containing $\mathcal F_j$ is within $\sqrt 3V_j$ in a three dimensional space \cite{Boyd2004}. Hence the volume $V_j$ can be used as a proxy measure for the volume of $\mathcal F_j$. Moreover, the problem of determining the maximum volume ellipsoid contained in $\mathcal F_j$ is convex while the problem of determining the minimum volume ellipsoid containing $\mathcal F_j$ is intractable given the inequality constraints \cite{Todd2016}. Furthermore, the estimate of the center of this ellipsoid is invariant under affine transformations unlike the Chebychev center \cite{Boyd2004}.  

Any $\bm x \in \mathcal E_j$ can be written as $\bm x = P_j\bm u + \bm r_{c,j}$, where $P_j \succeq 0$ (symmetric and positive semi-definite matrix of appropriate dimension), $\bm r_{c,j}$ is the center of the $\mathcal E_j$ and $\norm{\bm u}^2 \leqslant 1$.  Using these notations, the problem of computing the maximum volume inscribed ellipsoid $\mathcal E_j $ of $\mathcal F_j$ can be formulated as a SDP:
\begin{flalign}
(\mathcal L_3) \quad V_j := &\max \log \operatorname{det}~P_j,\notag \\
\text{~subject to: } P_j &\succeq 0 \text{ and} \label{eq:mve-sdp} \\
\norm{P_j\bm u + \bm r_{c,j} - \bm B_i} &\leqslant \phi(D_{ij}), \notag \\
\qquad \qquad &\forall i \in I, \forall \{\bm u:\norm{\bm u}^2 \leqslant 1\}. \label{eq:mve-spheres} 
\end{flalign} 
The constraints \eqref{eq:mve-spheres} are convex and represent the ellipsoid $\mathcal E_j$ that is constrained to lie in the intersection of the spheres centered at the beacon locations of radii determined by the range measurements gathered by the $j$\textsuperscript{th} receiver.
\begin{proposition} \label{prop:mve}
For a fixed value of $i \in I$, the set of infinite constraints \eqref{eq:mve-spheres} in formulation $\mathcal L_3$ is equivalent to the following set of constraints:
\begin{flalign} \lambda \geqslant 0 \text{ and }
\begin{bmatrix}
\phi(D_{ij})-\lambda & (\bm r_{c,j}-\bm B_i)^T & 0 \\ 
(\bm r_{c,j} - \bm B_i) & \phi(D_{ij}) I_3 & P_j \\ 
0 & P_j & \lambda I_3
\end{bmatrix} \succeq 0.
\end{flalign}
\end{proposition}
\begin{proof}
See Appendix A.
\end{proof}
\noindent Hence, an equivalent formulation for computing $\mathcal E_j$ is given by:
\begin{flalign*}
& (\mathcal L_4) \quad V_j = \max \log \operatorname{det}~P_j,\text{~subject to: } & \\ 
&\begin{bmatrix}
\phi(D_{ij})-\lambda & (\bm r_{c,j}-\bm B_i)^T & 0 \\ 
(\bm r_{c,j} - \bm B_i) & \phi(D_{ij}) I_3 & P_j \\ 
0 & P_j & \lambda I_3
\end{bmatrix} \succeq 0, P_j \succeq 0, \text{ and } \lambda \geqslant 0. 
\end{flalign*}

\subsection{Determination of the calibration function $\phi(\cdot)$ for the sensing model \label{subsec:calibration}}
As stated in the aforementioned section, given that the measured distance from a beacon $i$ to an the on-board receiver $j$, $D_{ij}$, the sensing model assumes that $\phi(D_{ij})$ is an increasing function that provides a bound on the true distance $d_{ij}$. One important ramification of choosing an increasing function $\phi(\cdot)$ is that the range measurements that originate from the beacons sufficiently far away from the vehicle would implicitly be ignored. The model also provides a bound on the true distance $d_{ij} \leqslant \phi(D_{ij})$ \emph{i.e.}, the model ensures that the true location of on-board receiver always lies within the disc centered at the beacon with radius $\phi(D_{ij})$. In this section, we formulate the problem of finding an increasing function $\phi(\cdot)$ using the range measurements and true distance data as a SDP. This problem can either be solved off-line before the beacons and the on-board receivers are deployed (receiver calibration process before deployment) or can be adaptively estimated and updated if there are other sensors in place to estimate the position of the vehicle via the technique of sensor fusion. For ease of exposition, we will first develop the mathematical formulation in an off-line setting where prior knowledge of the environment in which the AUV is to be operated is assumed. Adaptive, online estimation of the calibration function is presented later in section \ref{subsec:phi}.

We let $K$ denote a set of distances between a particular beacon and a receiver; this set is constructed with a prior knowledge of the environment in which the AUV and the beacons are deployed. For each data point $d_k$, $k\in K$, multiple range measurements are obtained while placing an arbitrary beacon and an on-board receiver $d_k$ units apart; we note that each of these range measurements corresponding to the true distance $d_k$ is noisy. Let $[D_k^l, D_k^u]$ denote the lower and upper bounds of the measurements corresponding to the true distance $d_k$, $k \in K$. Define $D^l := \min_{k\in K} D_k^l$ and $D^u := \max_{k\in K} D_k^u$. The problem now is to determine a univariate increasing function $\phi: [D^l, D^u] \rightarrow \mathbb{R}^+$ that solves the following optimization problem: 
\begin{flalign}
\min \sum_{k \in K} &(\phi(D_k^u) - d_k), \text{ subject to:} \label{eq:phi-obj}\\ 
\phi'(d) &\geqslant 0, ~\forall d\in [D^l, D^u], \text{ and} \label{eq:phi1} \\
\phi(D_k^l) &\geqslant d_k, ~\forall k\in K.\label{eq:phi2}
\end{flalign}  
The constraints \eqref{eq:phi1} enforce the function $\phi(\cdot)$ to be an increasing function in its domain and the constraints \eqref{eq:phi2} ensure for any range measurement $D_{ij}$ between a beacon $i$ and an on-board receiver $j$, $\phi(D_{ij})$ provides an upper bound on the true distance $d_{ij}$. The objective ensures that the function $\phi(\cdot)$ provides the tightest bound in the the sense of constraint \eqref{eq:phi2}. To solve the above problem, we approximate $\phi(x)$ using a univariate polynomial. The condition \eqref{eq:phi1} is equivalent to the polynomial $\phi'(\cdot)$ being non-negative in the interval $[D^l, D^u]$. We now state two known results using which we recast the non-negativity restrictions on the univariate polynomial to a SDP.  
\begin{theorem} \label{thm:ML}
(Markov-Luk\`acs theorem) Let $a < b$. Then a univariate polynomial $p(x)$ is non-negative on $[a,b]$, if and only if it can be written as 
\begin{flalign*}
& p(x) = 
\begin{cases}
s(x) + (x-a)(b-x)t(x), &\text{if } \operatorname{degree}(p) \text{ is even} \\
(x-a)s(x) + (b-x)t(x), &\text{if } \operatorname{degree}(p) \text{ is odd}
\end{cases}
\end{flalign*}
where $s(x)$ and $t(x)$ are `Sum of Squares' (SOS). In the first case, we have $\operatorname{degree}(p) = 2k$, and $\operatorname{degree}(s) \leq 2k$, $\operatorname{degree}(t) \leq 2k-2$. In the second, $\operatorname{degree}(p) = 2k+1$, and $\operatorname{degree}(s) \leq 2k$, $\operatorname{degree}(t) \leq 2k$.
\end{theorem}
\begin{proof}
	See \cite{Polya1976}. \qed
\end{proof}
\begin{theorem} \label{thm:sos}
Let $p(x)$ be a univariate polynomial of degree $2k$. Then, $p(x)$ is SOS if and only if there exists a $(k+1)\times (k+1)$ positive semi-definite matrix $\bm P$ that satisfies $$p(x) = [x]_k^T\bm P[x]_k$$ where, $[x]_k = [1~x~x^2 \dots x^k]^T$.
\end{theorem}
\begin{proof}
See \cite{Parrilo2003}.
\end{proof}
Let $\phi(x) = a_0 + a_1 x + \dots + a_n x^n$ be an $n$\textsuperscript{th} degree polynomial whose coefficients are to be determined. Then, we have
\begin{flalign}
	\phi'(x) = a_1 + 2a_2 x + \dots na_n x^{n-1} = \sum_{i=1}^n ia_i x^{i-1}. \label{eq:phi-derivative}
\end{flalign} 
The constraint \eqref{eq:phi1} requires $\phi'(x)$ to be non-negative on $[D^l, D^u]$. Suppose that the degree of $\phi$ is even (the case when $\phi$ has an odd degree has a similar reduction and hence, is not presented). Then by theorem \ref{thm:ML}, we have 
\begin{flalign}
	\phi'(x) = (x-D^l) s(x) + (D^u-x) t(x) \label{eq:phi-ml}
\end{flalign}
where, $s(x)$ and $t(x)$ are SOS. 
Let $\operatorname{degree}(\phi') = 2k+1$, then degree of $s(x)$ and $t(x)$ is at most $2k$. Then by theorem \ref{thm:sos}, we have 
\begin{flalign}
	s(x) = [x]_k^T \bm S [x]_k,~ t(x) = [x]_k^T \bm T [x]_k \label{eq:sos} 
\end{flalign}
where, $\bm S$ and $\bm T$ are $(k+1) \times (k+1)$ positive semi-definite matrices. Combining \eqref{eq:phi-derivative}, \eqref{eq:phi-ml} and \eqref{eq:sos}, we obtain:
\begin{flalign*}
	\sum_{i=1}^{2k+1} ia_i x^{i-1} = (x-D^l)[x]_k^T \bm S [x]_k + (D^u-x)[x]_k^T \bm T [x]_k.
\end{flalign*}
Indexing the rows and columns of $\bm S$ and $\bm T$ by $\{0,1,\dots,k \}$ and equating the coefficients of the $M$\textsuperscript{th} power of $x$ on both sides of the above equation, we obtain a set of $(2k+1)$ linear equations relating the coefficients of the polynomial $\phi(\cdot)$ and then entries of the positive semi-definite matrices $\bm S$ and $\bm T$ as follows:
\begin{flalign}
	(M+1) a_{M+1} = \sum_{0\leqslant i,j \leqslant k}^{i+j = M-1} \left(\bm Q_{ij} - \bm T_{ij}\right) + \notag \\
	\qquad \qquad \sum_{0\leqslant i,j \leqslant k}^{i+j = M} \left(D^u \bm T_{ij} - D^l\bm S_{ij}\right). \label{eq:coeff}
\end{flalign}
Hence, an equivalent SDP for computing the function $\phi$, which is approximated by a degree $n$ polynomial is given by:
\begin{flalign*}
&(\mathcal L_5) ~ \min ~\sum_{k \in K} (\phi(D_k^u) - d_k), &\\
& \text{ subject to: } \eqref{eq:coeff}, \eqref{eq:phi2}, \text{ and } \bm S, \bm T \succeq 0. &
\end{flalign*}  
Constraints \eqref{eq:coeff} and the positive semi-definiteness of $\bm S, \bm T$ together is equivalent to enforcing the polynomial $\phi$ to be increasing; the constraints \eqref{eq:coeff} and \eqref{eq:phi2} are linear constraints.
%
%
%
\subsection{Estimation of the orientation of the vehicle from the location estimates of the on-board receivers \label{subsec:orietation}}
Suppose $\mathcal F$ is a frame of reference attached to the rigid body with its origin at $O$ and unit vectors $\hat{\imath},\hat{\jmath},\hat{k}$, respectively. Let the coordinates of the vehicle's on-board receivers in $\mathcal F$ be $\bm w_j = (a_j,b_j,c_j),$ $j \in J$, respectively and its estimated location in the ground frame be $\bm r_{c,j} = (x_{c,j},y_{c,j},z_{c,j})$. Let $\bm R$ be the rotation matrix associated with the body describing its orientation. Let $\bm r_0 = (x_0,y_0,z_0)$ denote the estimate of the location of the origin $O$ of the body frame $\mathcal F$. Then, it is clear that the following rigid body motion constraints must hold when there is no estimation error in the location of the on-board receivers:
\begin{flalign}
\bm r_{c,j} = \bm r_0 + \bm R \bm w_j, ~\forall j\in J. \label{eq:rigidbodyconstraints-a}
\end{flalign}
Essentially, these constraints guarantee that the angles between line segments joining the receivers as inferred from the location estimates will remain the same as their true values and the distance between the receivers as inferred from their locations will remain the same as their true values. Compactly, one can rewrite the above equation as:
\begin{flalign}
\begin{bmatrix}
\bm r_{c,1} \cdots \bm r_{c,L}
\end{bmatrix} = 
\begin{bmatrix}
\bm r_0 \cdots \bm r_0 
\end{bmatrix} + \bm R 
\begin{bmatrix}
\bm w_1  \cdots \bm w_L
\end{bmatrix}. \label{eq:rigidbodyconstraints-b}
\end{flalign}
However, the estimates may not satisfy the above relationship due to errors in measurements and subsequent location estimation of on-board sensors. In particular, the estimate of the distance between the on-board receivers need not equal the actual distance between them. As a consequence, the relative configuration of the on-board receivers indicated by their location estimates will not be the same as the true relative configuration of receivers. One then needs to correct these location estimates in order to ensure that the distance between the on-board receivers is its true value. Since the errors in the location estimates will be non-zero, let us define an error matrix, $\bm E$ as:
\begin{flalign}
\bm E:= 
\begin{bmatrix}
\bm r_{c,1} \cdots \bm r_{c,L}
\end{bmatrix} -
\begin{bmatrix}
\bm r_0 \cdots \bm r_0 
\end{bmatrix} - \bm R 
\begin{bmatrix}
\bm w_1  \cdots \bm w_L
\end{bmatrix}. \label{eq:error}
\end{flalign}
The problem of localization can now be posed as 
\begin{flalign*}
(\mathcal L_6) \quad J = \min_{\bm r_0 \in \Re^3, \bm R \in SO(3)} \operatorname{trace}(\bm E^T \bm E) \text{ subject to \eqref{eq:error}.}	
\end{flalign*}
In the above formulation $\mathcal L_6$, $\operatorname{trace}(\bm E^T \bm E)$ is the square of the Frobenius norm of the error matrix $\bm E$, and $SO(3) = \{\bm R: \operatorname{det}(\bm R) = 1, \bm R^{-1} = \bm R^T\}$ is referred to as the `Special Orthogonal Group'.

\section{Algorithms \label{sec:algorithms}}
In this section, we focus on solving the relevant optimization problems from the previous section. In section \ref{subsec:location}, we will outline a cutting plane algorithm to solve the formulations $\mathcal{L}_2$. In section \ref{subsec:orientation}, we will provide a solution procedure to solve the problem of localization given by the formulation $\mathcal L_6$ and thereby determine the optimal orientation that minimizes the square of the Frobenius norm of the error matrix, $\bm E$. As for the formulation $\mathcal L_4$ and $\mathcal L_5$, they can be solved to optimality using off-the-shelf semi-definite solvers like SCS \cite{SCS2013}.

\subsection{Location estimation procedure \label{subsec:location}}
\subsubsection{Algorithm to estimate the Chebychev center \label{subsubsec:CC-CP}}
The procedure involves a relaxation of the semi-infinite LP in the formulation $\mathcal L_2$ to a finite LP by ignoring all but finite constraints and providing an iterative way of adding the required constraints from the dropped set of constraints. This generic procedure is referred to as a cutting-plane method (see \cite{Boyd2004}).

To that end, let $\bm v_1,\dots,\bm v_M$ be unit vectors representing the $M$ sides of a circumscribing polygon of the feasible region, then a relaxation of $\mathcal L_2$ is given by
\begin{flalign*}
&\bar{l}_{\max} = \max_{l, \bm r_{c,j}} ~ l, \text{ subject to: } & \\
& \bm v_k \cdot \bm r_{c,j} + l  \norm{\bm v_k} \leqslant  \bm v_k \cdot \bm B_i + \phi(D_{ij}), & \\
& \qquad \qquad \qquad \qquad \qquad ~\forall i \in I,~ k=1,..,M. &
\end{flalign*}
Clearly, the feasible set of this LP, $\bar{\mathcal F}_{j}$, contains the feasible set $\mathcal F_j$ of the original problem as all by finite constraints of the original semi-infinite LP have been dropped. Suppose the optimal solution, $(\bar{l}_{\max}, \bar{\bm r}_{c,j})$, of the relaxed finite LP satisfies semi-infinite constraints; then it is clear that $(\bar{l}_{\max}, \bar{\bm r}_{c,j})$ is optimal for $\mathcal L_2$. Otherwise, for some unit vector $\bm v_{M+1}$ distinct from those considered before, and for some $i$, the following inequality holds: $$ \bm v_{M+1} \cdot \bar{\bm r}_{c,j} + \bar{l}_{\max}  >  \bm v_{M+1} \cdot \bm B_i + \phi(D_{ij}).$$ Such a unit vector $\bm v_{M+1}$ can be easily computed by maximizing the function $\left[\bm v \cdot (\bar{\bm r}_{c,j} - \bm B_i)\right]$ over the unit ball $\norm{\bm v} = 1$ and checking if the optimum objective is strictly greater than $\phi(D_{ij}) - \bar{l}_{\max}$. The maximization problem can trivially be solved using Cauchy-Schwartz inequality. By adding the ``cut'' $$ \bm v_{M+1} \cdot \bm r_{c,j} + l \leqslant  \bm v_{M+1} \cdot \bm B_i + \phi(D_{ij}),$$ which must be satisfied by the optimal solution for the semi-infinite LP and violated by the previously obtained optimal solution for the finite LP, we improve the solution. This is akin to finding another face of the polygon circumscribing the disk that cuts off a vertex of the previously obtaining polygon. This cutting plane method can be used along with an off-the shelf LP solver to solve the semi-infinite LP to arbitrary accuracy.

\subsection{Procedure for orientation estimation and correcting the location estimates taking into account the rigid body constraints} \label{subsec:orientation}
The location estimates of the on-board receivers have been obtained without regard to the rigid body motion constraints between them. Since receivers are attached to the rigid body, the distance between any pair of them is pre-specified. The estimates may not satisfy the distance constraints and even the angle between line segments joining the receivers computed from their location estimates may not correspond to their true values. For this reason, a correction procedure for the location estimates is required. Fortunately, this pursuit involves the estimation of orientation of the body. 

Let, $$ \bm e_j := \bm r_{c,j} - \bm r_0 - \bm R \bm w_j, ~j=1,\dots,L. $$ The term $\bm e_j$ describes the error in the estimate of the location to maintain a rigid body constraint for the $j$\textsuperscript{th} receiver and is a function of the the location $\bm r_0$ of the origin of the body frame and the rotation matrix $\bm R$ that describes the orientation of the rigid body. One can observe that $\operatorname{trace}(\bm E^T \bm E) = \sum_{j=1}^L \bm e^T_j \bm e_j$ and $\bm e_j$ is a linear function of $\bm r_0 = (x_0 , y_0 , z_0)$ and $\bm R$. Hence, minimization over $\bm r_0, \bm R$ can be performed sequentially since there are no constraints if we explicitly express $\operatorname{trace}(\bm E^T \bm E)$ as a function of $\bm r_0, \bm R$. Let $$ \bar{\bm w} = \frac 1L \sum_{j=1}^L \bm w_j, ~ \bar{\bm r} = \frac 1L \sum_{j=1}^L \bm r_{c,j}.$$ Minimization of $\operatorname{trace}(\bm E^T \bm E)$ with respect to $\bm r_0$ yields $$ \bm r_0 = \frac 1L \sum_{j=1}^L \left[ \bm r_{c,j} - \bm R \bm w_j\right] = \bar{\bm r} - \bm R \bar{\bm w}.$$ Define for $j\in J$ $$ \tilde{\bm r}_{c,j} := \bm r_{c,j} - \bar{\bm r}, ~ \tilde{\bm w}_j := \bm w_j - \bar{\bm w}. $$ With these definitions and the optimizing value of $\bm r_0$, $$ \bm e_j = \tilde{\bm r}_{c,j} - \bm R \tilde{\bm w}_j, ~j \in J.$$ 
Correspondingly, 
\begin{flalign*}
	& \operatorname{trace}(\bm E^T \bm E) = \sum_{j=1}^L (\tilde{\bm r}_{c,j} - \bm R \tilde{\bm w}_j)^T (\tilde{\bm r}_{c,j} - \bm R \tilde{\bm w}_j) &\\
	& = \sum_{j=1}^L \left(\tilde{\bm r}_{c,j}^T \tilde{\bm r}_{c,j} + \tilde{\bm w}_j^T \tilde{\bm w}_j \right) - 2\operatorname{trace}\left(\left(\sum_{j=1}^L \tilde{\bm w}_j \tilde{\bm r}_{c,j}^T\right)\bm R\right).&
\end{flalign*}
Define $\bm W := \sum_{j=1}^L \tilde{\bm w}_j \tilde{\bm r}_{c,j}^T$ so that $$\bm R^* = \argmin_{\bm R \in SO(3)} \operatorname{trace}(\bm E^T \bm E) = \argmax_{\bm R \in SO(3)} \operatorname{trace}(\bm W\bm R). $$This stems from the other terms being independent of $\bm R$. 

Let the singular value decomposition of $\bm{W} = \bm{U\Sigma} \bm V^T$ where $\bm U$, $\bm V$ are the left and right singular vectors of $\bm W$, respectively and $\bm \Sigma$ is a $3 \times 3$ diagonal matrix consisting of its singular values. The problem of maximizing $\operatorname{trace}(\bm W \bm R)$ over the set of all rotation matrices is referred to as the ``Orthogonal Procrustes problem'' (see \cite{Schoenemann1966}).

\begin{theorem} 
$\bm R^* = \bm V \bm U^T$ maximizes $\operatorname{trace}(\bm W \bm R)$ over the set of all proper rotation matrices.
\end{theorem}
\begin{proof} We have due to the property of $\operatorname{trace}$ of a product of matrices:
\begin{flalign*}
	\operatorname{trace}(\bm W \bm R) = \operatorname{trace} (\bm U \bm \Sigma \bm V^T \bm R) = \operatorname{trace}(\bm U \bm V^T \bm R \bm \Sigma).
\end{flalign*}	
Note that $\bm V$, $\bm R$ and $\bm U$ are all orthonormal matrices, so $\bm X = \bm U \bm V^T \bm R$ is also an orthonormal matrix. Suppose $x_{ii}$, $i=1,2,3$ denotes the diagonal elements of $\bm X$, then 
\begin{flalign*}
	\operatorname{trace}(\bm W \bm R) = \operatorname{trace}(\bm X \bm \Sigma ) = \sum_{i=1}^3 \sigma_i x_{ii} \leq \sum_{i=1}^3 \sigma_i
\end{flalign*}
where, $\sigma_i$, $i=1,2,3$ are the singular values of $\bm W$. The maximum value in the above equation is achieved when $\bm X = \bm U \bm V^T \bm R = \bm I$ \emph{i.e.}, $\bm R^* = \bm V \bm U^T$.
\end{proof}

\begin{remark} The minimum value of $$\operatorname{trace}(\bm E^T \bm E) = \sum_{j=1}^L \left(\tilde{\bm r}_{c,j}^T \tilde{\bm r}_{c,j} + \tilde{\bm w}_j^T \tilde{\bm w}_j \right) - 2(\sigma_1 + \sigma_2 + \sigma_3).$$
\end{remark}

\begin{remark}
	The location estimate of the origin of the body frame is $\bm r_0 = \bar{\bm r} - \bm R^* \bar{\bm w}.$ If one places the on-board receivers in such a way that the ``center of mass'' of the rigid body (vehicle) coincides with the center of mass of the receivers, then the estimate of $\bm r_0$ is the ``corrected'' estimate of the location of the center of mass. The error in localization (both in the location estimation and orientation) from the ``best'' estimates of the locations of the receiver is given by the sum of the singular values of $\bm W$, a norm referred to as the nuclear norm of $\bm W$.
\end{remark}

\begin{remark}
	The updated estimate of the $j$\textsuperscript{th} receiver's location will be given by: $$\bm r_{c,j} = \bm r_0 + \bm R^* \bm w_j.$$
\end{remark}

\section{Computational and experimental results} \label{sec:results}
This section presents numerical simulations and experimental results to corroborate the effectiveness of the algorithms proposed in this article. All the simulations were performed on a Dell Precision Workstation with 12 GB RAM. The simulation setup is as follows: We assume that a set of LBL acoustic beacons are fixed underwater at predetermined locations (by either anchoring them to sea floor, or by suspending GIBs from the surface) in a manner that the vehicle's route is completely contained in the convex hull of the positions of the acoustic beacons. The algorithms presented in this article are valid only under this assumption. The coordinates of the positions of the beacons with respect to an inertial frame of reference is assumed to be known a priori. The beacon locations and the convex hull of the beacon positions are shown in the figure \ref{fig:beacons}. The AUV to be localized is equipped with four on-board receivers that are fixed at the origin and the three unit vectors of the vehicle reference frame; without loss of generality, the origin of the vehicle reference frame is assumed to be located at the center of mass of the vehicle. Now, given the range measurements of all the on-board receivers from the beacons at various time instances, the problem is to determine the position of the center of mass of the vehicle and the vehicle's orientation. We remark that though the simulation setup might not correspond to a realistic scenario, we use such a setup to test the robustness and strength of the method; hence, in practice, our algorithms can have a better performance than the ones presented in the numerical simulations. Nevertheless, experimental results that corroborate the effectiveness of the proposed algorithms are also presented. 

\begin{figure}
\centering
\includegraphics[scale=0.5]{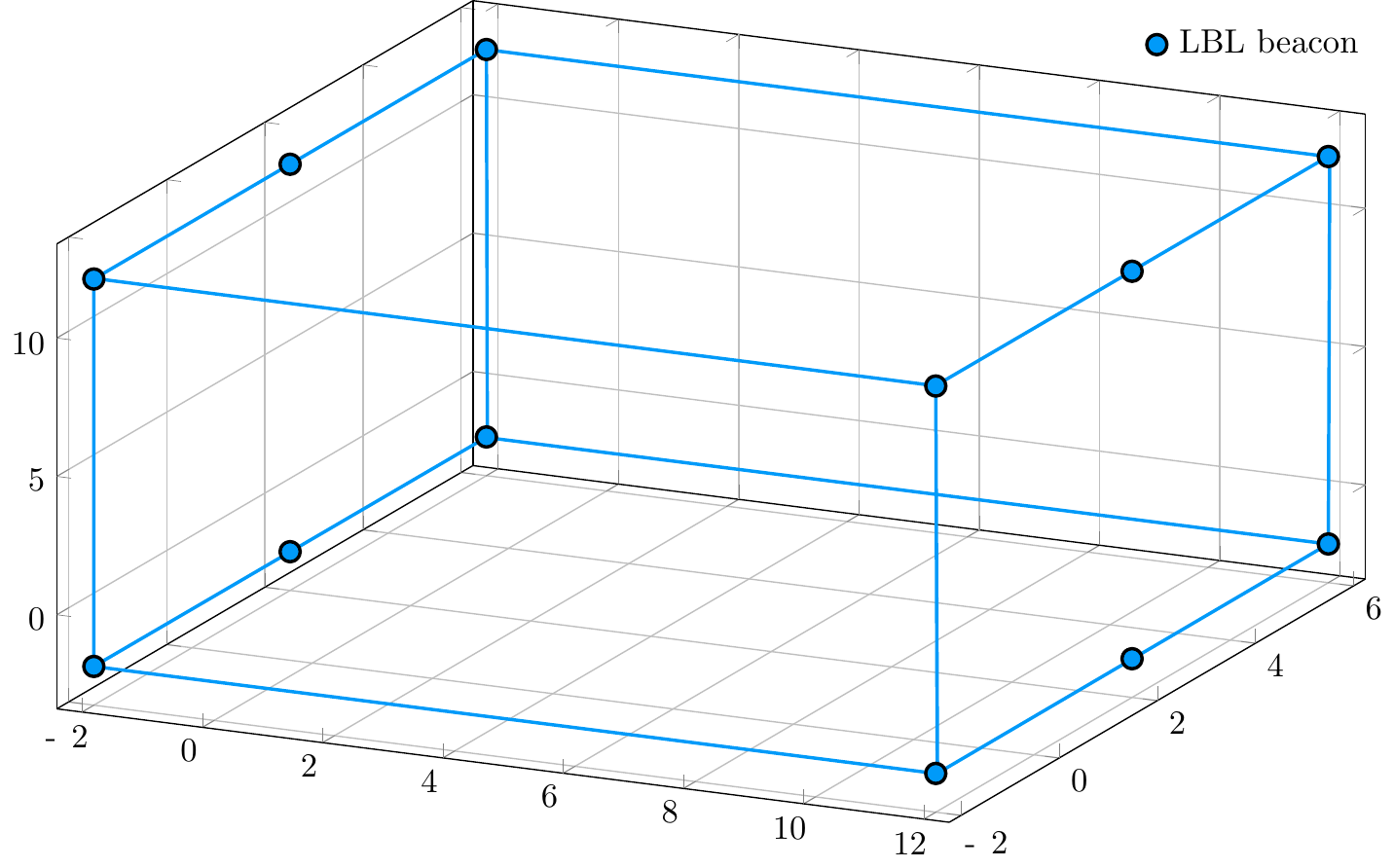}
\caption{Beacon positions are shown. The vehicle is assumed to move along a path completely contained in the convex hull.}
\label{fig:beacons}
\end{figure}

\subsection{Simulation of range measurements} \label{subsec:range_simulation}
In reality, the range measurements are available from the {One Way Travel Time (OWTT) of pulses} that are transmitted from the beacons and received by the on-board receivers on the vehicle. In the current example, these range values are numerically simulated using the following procedure: the vehicle is assumed to follow a particular trajectory in which the coordinates of the on-board receiver located at the center of mass of the vehicle at any time instant $t$ is given by $(2.5(1+\cos t), 2.5 \sin t , 5 \sin(t/2))$; the other three on-board receivers are assumed to move along the tangent, normal, and bi-normal directions of the curve. These directions are unique at each instant of time and are calculated using the Frenet-Serret equations. The range measurements from each beacon to each on-board receiver is then generated by adding a white Gaussian noise with mean $0$ and variance $0.25$ to the true values. $100$ such measurements for each beacon-receiver pair are generated as the vehicle travels along its path. Using these range measurements the estimate of the position of each on-board receiver is computed using the algorithm presented in \ref{subsec:location}. The estimate of the position of the vehicle is then provided by the position estimate of the on-board receiver located at the center of mass of the vehicle. In the following section, we shall analyze the performance of the position estimation algorithms. 

\subsection{Position estimation algorithm performance}
$\phi(\cdot)$ enables the conversion of range measurements obtained using each beacon to the radii of the spheres centered on the beacons and that contain location of the on-board receiver. For the purpose of this section, we assume that the function $\phi(\cdot)$ is computed offline. The intersection of all the spheres is the set of all possible receiver locations. For the purpose of analysing the performance of the position estimation algorithms, we concern ourselves only with the position of the on-board receiver located at the center of mass of the vehicle. The Chebychev center of the feasible region is computed using the cutting-plane algorithm that is proposed in section \ref{subsubsec:CC-CP} using an LP for each set of range measurements and the center of the maximum volume ellipsoid that can be inscribed in this feasible region is computed directly by solving an SDP, derived in section \ref{subsubsec:mve}, using SCS \cite{SCS2013} -- an off-the-shelf commercial SDP solver.

The average number of cuts, over all the $100$ runs, added by the cutting-plane algorithm to compute the Chebychev center was $29$ and the average computation time was $0.0079$ seconds. The figure \ref{fig:cheb_error} shows the error, as a percentage of the maximum range measurement, in the position estimate of the center of the mass of the vehicle using the cutting plane algorithm for the Chebychev center. The average error was found to be $1.55\%$ and the maximum error was $5.73\%$.

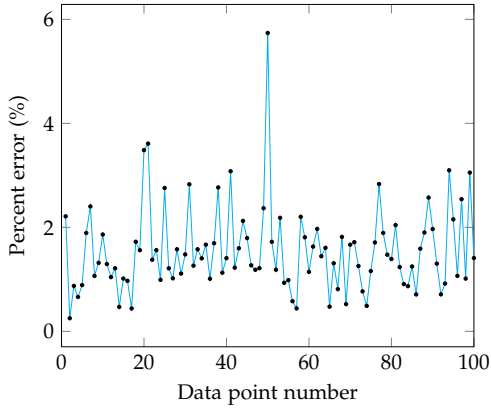
\begin{figure}[ht]
\centering
\begin{tikzpicture}[scale=0.8]
\begin{axis}[xlabel = {Data point number}, ylabel = {Percent error ($\%$)}, xmax = {100}, xmin = {0}]\addplot+ [cyan, mark options={black}, mark size = {0.8}]coordinates {
(1.0, 2.21195808)
(2.0, 0.2517366)
(3.0, 0.87413181)
(4.0, 0.66058848)
(5.0, 0.88902752)
(6.0, 1.8925786)
(7.0, 2.40228758)
(8.0, 1.06723756)
(9.0, 1.31817135)
(10.0, 1.86235598)
(11.0, 1.29460017)
(12.0, 1.04267698)
(13.0, 1.2128971)
(14.0, 0.47036843)
(15.0, 1.0155559)
(16.0, 0.97173525)
(17.0, 0.44016109)
(18.0, 1.72276375)
(19.0, 1.5607416)
(20.0, 3.48489747)
(21.0, 3.61077762)
(22.0, 1.37796333)
(23.0, 1.56195097)
(24.0, 0.9907436)
(25.0, 2.75737692)
(26.0, 1.21236498)
(27.0, 1.01971911)
(28.0, 1.57868761)
(29.0, 1.11129289)
(30.0, 1.4795011)
(31.0, 2.82792563)
(32.0, 1.26333262)
(33.0, 1.57831418)
(34.0, 1.40318617)
(35.0, 1.66659893)
(36.0, 1.01168947)
(37.0, 1.69439743)
(38.0, 2.76690026)
(39.0, 1.12755008)
(40.0, 1.40755847)
(41.0, 3.07873944)
(42.0, 1.22534658)
(43.0, 1.59769614)
(44.0, 2.12406759)
(45.0, 1.79362739)
(46.0, 1.26922725)
(47.0, 1.18406261)
(48.0, 1.21485733)
(49.0, 2.36834044)
(50.0, 5.73847009)
(51.0, 1.72139739)
(52.0, 1.18495468)
(53.0, 2.18352915)
(54.0, 0.93400543)
(55.0, 0.98547804)
(56.0, 0.57955395)
(57.0, 0.44127922)
(58.0, 2.2022525)
(59.0, 1.80882182)
(60.0, 1.14455568)
(61.0, 1.63025878)
(62.0, 1.96999362)
(63.0, 1.44427213)
(64.0, 1.60576943)
(65.0, 0.47481402)
(66.0, 1.30916851)
(67.0, 0.81307573)
(68.0, 1.81645253)
(69.0, 0.52378164)
(70.0, 1.66686233)
(71.0, 1.71504118)
(72.0, 1.25614789)
(73.0, 0.76882191)
(74.0, 0.48992216)
(75.0, 1.1594055)
(76.0, 1.70930122)
(77.0, 2.83279454)
(78.0, 1.89235579)
(79.0, 1.4748645)
(80.0, 1.39140498)
(81.0, 2.04322789)
(82.0, 1.23590536)
(83.0, 0.90885416)
(84.0, 0.86908438)
(85.0, 1.24728166)
(86.0, 0.71097995)
(87.0, 1.59000984)
(88.0, 1.90075172)
(89.0, 2.5729736)
(90.0, 1.96710859)
(91.0, 1.30182683)
(92.0, 0.71092605)
(93.0, 0.91902795)
(94.0, 3.09687985)
(95.0, 2.15293375)
(96.0, 1.0680822)
(97.0, 2.54124566)
(98.0, 1.01490321)
(99.0, 3.05325883)
(100.0, 1.40990626)
};
\end{axis}

\end{tikzpicture}
\caption{Percentage error in the position estimate  of the center of mass of the vehicle computed using the cutting-plane algorithm for the Chebyshev center.}
\label{fig:cheb_error}
\end{figure}

As for the position estimate computed as the maximum volume ellipsoid center, the average computation time was found to be $0.14$ seconds. The figure \ref{fig:mve_error} shows the percentage error in the position estimate of the center of mass of the vehicle obtained by solving the SDP for the maximum volume ellipsoid center. The average error was found to be around $1.42\%$ and the maximum error was found to be $3.25\%$. 

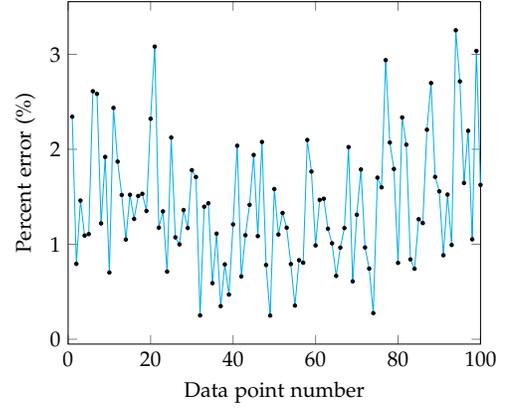
\begin{figure}[ht]
\centering
\begin{tikzpicture}[scale=0.8]
\begin{axis}[xlabel = {Data point number}, ylabel = {Percent error ($\%$)}, xmax = {100}, xmin = {0}]\addplot+ [cyan, mark options={black}, mark size = {0.8}]coordinates {
(1.0, 2.34362529)
(2.0, 0.79364334)
(3.0, 1.46047511)
(4.0, 1.09167615)
(5.0, 1.1070838)
(6.0, 2.61119814)
(7.0, 2.58380249)
(8.0, 1.21980059)
(9.0, 1.91995419)
(10.0, 0.70285695)
(11.0, 2.43598355)
(12.0, 1.87030063)
(13.0, 1.51950297)
(14.0, 1.0494994)
(15.0, 1.52101438)
(16.0, 1.26728953)
(17.0, 1.50863653)
(18.0, 1.53137018)
(19.0, 1.35055129)
(20.0, 2.32123018)
(21.0, 3.08088821)
(22.0, 1.17482525)
(23.0, 1.3464163)
(24.0, 0.71189919)
(25.0, 2.12370411)
(26.0, 1.07258646)
(27.0, 0.99888507)
(28.0, 1.36064868)
(29.0, 1.17177805)
(30.0, 1.78042328)
(31.0, 1.70798611)
(32.0, 0.25105977)
(33.0, 1.39563229)
(34.0, 1.43190577)
(35.0, 0.58987432)
(36.0, 1.11184784)
(37.0, 0.34890105)
(38.0, 0.78779468)
(39.0, 0.47033817)
(40.0, 1.20912533)
(41.0, 2.03768445)
(42.0, 0.66051453)
(43.0, 1.09508924)
(44.0, 1.41509929)
(45.0, 1.94070594)
(46.0, 1.08556002)
(47.0, 2.07742547)
(48.0, 0.78210655)
(49.0, 0.24928608)
(50.0, 1.5817989)
(51.0, 1.10315734)
(52.0, 1.32943585)
(53.0, 1.17454319)
(54.0, 0.79164579)
(55.0, 0.35444925)
(56.0, 0.83198245)
(57.0, 0.80570063)
(58.0, 2.09854119)
(59.0, 1.76547919)
(60.0, 0.98808273)
(61.0, 1.46780226)
(62.0, 1.4796922)
(63.0, 1.16357258)
(64.0, 1.01016833)
(65.0, 0.66754753)
(66.0, 0.96491498)
(67.0, 1.17056592)
(68.0, 2.02339995)
(69.0, 0.60873669)
(70.0, 1.31061146)
(71.0, 1.78761512)
(72.0, 0.96631921)
(73.0, 0.74397348)
(74.0, 0.27525533)
(75.0, 1.70131673)
(76.0, 1.59968514)
(77.0, 2.93902118)
(78.0, 2.07135269)
(79.0, 1.79302399)
(80.0, 0.80370905)
(81.0, 2.33545798)
(82.0, 2.04871487)
(83.0, 0.840395)
(84.0, 0.7422644)
(85.0, 1.26413576)
(86.0, 1.22311313)
(87.0, 2.20603018)
(88.0, 2.69710583)
(89.0, 1.70944474)
(90.0, 1.55738339)
(91.0, 0.88406251)
(92.0, 1.52349991)
(93.0, 0.99243367)
(94.0, 3.25315205)
(95.0, 2.71492893)
(96.0, 1.64524249)
(97.0, 2.19502867)
(98.0, 1.05198892)
(99.0, 3.03562734)
(100.0, 1.62493542)
};
\end{axis}

\end{tikzpicture}
\caption{Percentage error in the position estimate of the center of mass of the vehicle computed using the SDP for the maximum volume ellipsoid center.}
\label{fig:mve_error}
\end{figure}

Though the time taken for the computing the Chebyshev center is at least $10$ times lesser than the maximum volume ellipsoid center, the error performance for the maximum volume ellipsoid center based position estimate is better and hence, throughout the rest of the article, we will use the maximum volume ellipsoid center as the estimate for the position of an on-board receiver.

\subsection{Determination of $\phi(\cdot)$} \label{subsec:phi}
The sensing model, as detailed in section \ref{subsec:calibration}, assumes that the receiver always lies in the intersection of spheres centered at the beacons from which the range measurements corresponding to the receiver are obtained. For the purpose of simulation, we assume that the function $\phi(\cdot)$ is a polynomial of degree $4$ \emph{i.e.}, $\phi(x) = a_0 + a_1 x + a_2 x^2 + a_3 x^3 + a_4 x^4 \Rightarrow \phi'(x) = a_1 + a_2 x + a_3 x^2 + a_4 x^3$. For computation of the coefficients, data sets containing plausible true range values and multiple noisy range measurements corresponding to each true value are required. Experimentally, these data sets can be obtained by placing the receivers at different distances and recording the beacon range measurements multiple times. For the purpose of simulation these values were uniformly chosen from the interval $[d-\varepsilon, d+\varepsilon]$, where $d$ is a true range value and $\varepsilon$ is the maximum error in range measurements.  The permissible true distance values are restricted to a certain interval based on the beacon locations; the true range values were restricted to the interval $[4, 18]$ units. A set of $25$ values uniformly distributed in this interval are picked to form the set of true range values, and for each element in the set, $100$ corresponding range measurements are chosen from the interval $[d-\varepsilon, d+\varepsilon]$, with an $\varepsilon$ value of 0.25 units; this simulation procedure is set up to mimic the error present in actuality. Once the data set is available, the coefficients of the fitting polynomial can be obtained as a solution of an SDP, as discussed in section \ref{subsec:calibration}

Using the results in section \ref{subsec:calibration}, the constraint \eqref{eq:phi1} for the fourth degree polynomial can be re-written as follows:
 \begin{flalign*}
 & (A(x))^2  (x-D^l)+ (B(x))^2 (D^u - x)  \geqslant 0, \text{where } & \\
 & A(x) = [x]_k^T \bm P [x]_k,~B(x) = [x]_k^T \bm Q [x]_k, & \\
 & [x]_k = [1~x~x^2 \dots x^k]^T, & \\
 & \bm P = \begin{bmatrix} p_{11} & p_{12} \\ p_{12} & p_{22} \end{bmatrix} \succeq 0 \text{ and } \bm Q = \begin{bmatrix} q_{11} & q_{12} \\ q_{12} & q_{22} \end{bmatrix} \succeq 0. &
 \end{flalign*}
The above set of constraints upon simplification yields the following equations :
\begin{flalign*}
& a_1 = - D^l p_{11} + D^u q_{11} &\\
& 2 a_2 = p_{11} - 2 D^l p_{12} + 2 D^u q_{12} - q_{11} &\\
& 3 a_3 = 2P_{12} - D^l p_{22} + D^u q_{22} - 2 q_{12}  &\\
& 4 a_4 = p_{22} - q_{22}. & 
\end{flalign*}
Furthermore, the constraint \eqref{eq:phi2} for the fourth degree polynomial $\phi(\cdot)$ can be written as 
 $$a_0 + a_1 D^l_k + a_2 (D^l_k) ^2 + a_3(D^l_k)^3 + a_4(D^l_k)^4  \geqslant d_k , \forall k \in K, $$ 
and the objective function is equivalent to minimizing the function  $$  \sum_{k \in K} (a_0 + a_1 D^u_k + a_2 (D^u_k) ^2 + a_3(D^u_k)^3 + a_4(D^u_k)^4 - d_k).$$

For the simulation set discussed above, solving the SDP using SCS as an off-the-shelf commercial solver yields a solution. The computation time is of the order of milliseconds implying that this computation can be performed online. The error performance for the position estimation algorithm using this $\phi(\cdot)$ function that was computed online is shown in figures \ref{fig:cheb_error} and \ref{fig:mve_error} for the Chebyshev center and the maximum volume ellipsoid center, respectively.

The function $\phi(\cdot)$ can also be computed online, \emph{adaptively}, in the absence of multiple range measurements for a true range value using the technique of sensor fusion. This technique usually requires other sensors in place to estimate the position of the on-board receiver of the vehicle; the position estimates from these other sensors can be used as data sets to compute the function $\phi(\cdot)$. In this case, the position estimation algorithm would either result in a huge error initially which improves over time or would be infeasible; when either of the two situations occur, the function $\phi(\cdot)$ can be recomputed using the values from the alternate sensory equipment in place as the true position estimate and the noise contaminated range measurement as the reported range measurement. The position estimation algorithm  becomes infeasible if the range measurements do not lie within the calibration limits $[D^l, D^u]$ in \eqref{eq:phi1}. If the algorithm is unable to estimate the position at a particular instant, it uses the next best available estimate at the same instant to recalibrate itself. Then, the algorithm continues to give estimates till it becomes infeasible again, and the process continues. The figure \ref{fig:recalib} shows the error performance of the algorithm with and without online re-computation of the function. Initially, the errors are equal because the function $\phi(\cdot)$ did not change until data point $42$. Once infeasibility occurs, the online re-computation process changes $\phi(\cdot)$ and improves the error performance. The average percent error ignoring the points of infeasibility was observed to be  $5.13\%$ when the function $\phi(\cdot)$ was recomputed online.

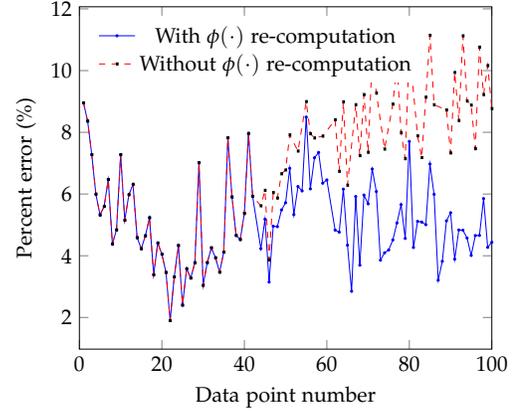
\begin{figure}[ht]
\centering
\begin{tikzpicture}[scale=0.8]
\begin{axis}[legend pos = {north west}, legend style={draw=none}, xlabel = {Data point number}, ylabel = {Percent error ($\%$)}, xmax = {100}, xmin = {0}]\addplot+ [mark size = {0.5}]coordinates {
(1.0, 8.95396663)
(2.0, 8.36237146)
(3.0, 7.27759947)
(4.0, 5.99496074)
(5.0, 5.32039522)
(6.0, 5.59862327)
(7.0, 6.4711425)
(8.0, 4.38300165)
(9.0, 4.83925516)
(10.0, 7.27662854)
(11.0, 5.15040659)
(12.0, 5.97954813)
(13.0, 6.31387461)
(14.0, 4.58586448)
(15.0, 4.22789749)
(16.0, 4.64762007)
(17.0, 5.23104906)
(18.0, 3.38480447)
(19.0, 4.41212826)
(20.0, 4.04967934)
(21.0, 3.45711782)
(22.0, 1.89792836)
(23.0, 3.3128417)
(24.0, 4.33089563)
(25.0, 2.3986077)
(26.0, 3.57906781)
(27.0, 3.27724877)
(28.0, 3.77497917)
(29.0, 7.01343945)
(30.0, 3.04167413)
(31.0, 3.77706144)
(32.0, 4.25578957)
(33.0, 3.93355063)
(34.0, 3.46855004)
(35.0, 4.11710025)
(36.0, 7.82333093)
(37.0, 5.90425228)
(38.0, 4.65861766)
(39.0, 4.52599002)
(40.0, 5.37123027)
(41.0, 7.9631301)
(42.0, 5.9260791)
(44.0, 4.22932698)
(45.0, 5.18355968)
(46.0, 3.14765706)
(47.0, 4.95614342)
(48.0, 4.94256969)
(49.0, 5.4837941)
(50.0, 5.71813718)
(51.0, 6.83220306)
(52.0, 5.33126812)
(53.0, 6.24258423)
(54.0, 6.10065838)
(55.0, 8.4940911)
(56.0, 6.16691451)
(57.0, 7.1753769)
(58.0, 7.346971)
(59.0, 6.35164803)
(60.0, 6.45611309)
(62.0, 4.83264464)
(63.0, 4.76458865)
(64.0, 6.15683506)
(65.0, 4.34146731)
(66.0, 2.84869808)
(67.0, 5.91841697)
(68.0, 3.69376926)
(69.0, 5.96537372)
(70.0, 5.68223267)
(71.0, 6.80980807)
(72.0, 6.0801976)
(73.0, 3.8573915)
(74.0, 4.09578759)
(75.0, 4.18947306)
(76.0, 4.51171982)
(77.0, 5.06112344)
(78.0, 5.65733288)
(79.0, 4.5611897)
(80.0, 7.70532613)
(81.0, 4.26692726)
(82.0, 5.11289666)
(83.0, 5.09185499)
(84.0, 5.00608147)
(85.0, 6.97350948)
(86.0, 5.99029233)
(87.0, 3.20770794)
(88.0, 3.81838235)
(89.0, 5.13189355)
(90.0, 5.39263168)
(91.0, 3.89058848)
(92.0, 4.83401135)
(93.0, 4.82624558)
(94.0, 4.57271763)
(95.0, 4.01208951)
(96.0, 4.65381641)
(97.0, 4.65784614)
(98.0, 5.85200679)
(99.0, 4.2726818)
(100.0, 4.43600932)
};
\addlegendentry{With $\phi(\cdot)$ re-computation}
\addplot+ [red, mark options={black}, dashed, mark size = {0.5}]coordinates {
(1.0, 8.95396663)
(2.0, 8.36237146)
(3.0, 7.27759947)
(4.0, 5.99496074)
(5.0, 5.32039522)
(6.0, 5.59862327)
(7.0, 6.4711425)
(8.0, 4.38300165)
(9.0, 4.83925516)
(10.0, 7.27662854)
(11.0, 5.15040659)
(12.0, 5.97954813)
(13.0, 6.31387461)
(14.0, 4.58586448)
(15.0, 4.22789749)
(16.0, 4.64762007)
(17.0, 5.23104906)
(18.0, 3.38480447)
(19.0, 4.41212826)
(20.0, 4.04967934)
(21.0, 3.45711782)
(22.0, 1.89792836)
(23.0, 3.3128417)
(24.0, 4.33089563)
(25.0, 2.3986077)
(26.0, 3.57906781)
(27.0, 3.27724877)
(28.0, 3.77497917)
(29.0, 7.01343945)
(30.0, 3.04167413)
(31.0, 3.77706144)
(32.0, 4.25578957)
(33.0, 3.93355063)
(34.0, 3.46855004)
(35.0, 4.11710025)
(36.0, 7.82333093)
(37.0, 5.90425228)
(38.0, 4.65861766)
(39.0, 4.52599002)
(40.0, 5.37123027)
(41.0, 7.9631301)
(42.0, 5.9260791)
(44.0, 5.61830977)
(45.0, 6.1189901)
(46.0, 3.86366934)
(47.0, 6.04695626)
(48.0, 5.87037805)
(49.0, 6.65426129)
(50.0, 6.77936684)
(51.0, 7.91466824)
(53.0, 7.39226532)
(55.0, 8.99545267)
(56.0, 7.95479863)
(57.0, 7.81614159)
(59.0, 7.87587474)
(61.0, NaN)
(62.0, 8.40817044)
(63.0, 6.73672844)
(64.0, 8.98913007)
(65.0, 6.28606445)
(67.0, 8.89614691)
(68.0, 7.24396502)
(69.0, 9.2215565)
(70.0, 7.35312309)
(71.0, 10.45093507)
(72.0, 9.27471183)
(74.0, 7.45667298)
(76.0, 8.91328399)
(77.0, 10.2651966)
(78.0, 7.99043683)
(79.0, 7.14458219)
(80.0, 10.71389756)
(82.0, 7.88580212)
(83.0, 7.18176609)
(84.0, 9.13902805)
(85.0, 11.14352496)
(86.0, 8.89370231)
(88.0, NaN)
(89.0, 8.72362781)
(90.0, 7.33053756)
(91.0, 9.93805057)
(92.0, 8.38261152)
(93.0, 11.12686566)
(94.0, 9.01304772)
(95.0, 8.88026457)
(96.0, 7.46896498)
(97.0, 10.76104922)
(98.0, 9.22298858)
(99.0, 10.16361934)
(100.0, 8.76924311)
};
\addlegendentry{Without $\phi(\cdot)$ re-computation}
\end{axis}

\end{tikzpicture}
\caption{Percentage error in the position estimate of the center of mass of the vehicle obtained using the maximum volume ellipsoid center using online $\phi(\cdot)$ estimation algorithm with and without re-computation}
\label{fig:recalib}
\end{figure}

\subsection{Errors in the the orientation estimates}
The receiver locations obtained from either the Chebyshev or maximum volume ellipsoid center estimates aid in estimating the positions of the all the $4$ on-board receivers; one in the center of mass of the vehicle and the other three along the unit vectors of the vehicle reference frame at unit distance from the origin of the frame. But these estimates need not conform to the rigid body constraints of the vehicle. Hence, these estimates are corrected using a orientation preserving rotation. This rotation is a proper orthogonal rotation computed via the solution to the Orthogonal Procrustes problem detailed in section \ref{subsec:orientation}. This rotation also provides an estimate of the orientation of the vehicle.  At each time step, the receiver locations were updated and the vehicle orientation is obtained as a rotation matrix $\bm R^{\dagger}$ using a singular value decomposition, as detailed in \ref{subsec:orientation}. Furthermore, in simulation, since we know the actual path the vehicle takes, we also know the actual orientation of the vehicle $\bm R$. The difference between the two rotations is then computed using a metric on $SO(3)$ that utilizes the idea of geodesic distance between $\bm{R}^{\dagger}$ and $\bm{R}$ \cite{Huynh2009} given by $$ d(\bm{R}, \bm{R}^{\dagger}) = \norm{ \operatorname{log}(\bm{R}^T \bm{R}^{\dagger})}_F$$ in $SO(3)$, where $\|\cdot\|_F$ denotes the Forbenius norm (see section \ref{subsec:orientation}); logarithm of a matrix is another matrix such that the matrix exponential of the latter matrix equals the original matrix. This distance is measured in radians. The figures \ref{fig:cheb_error_r} and \ref{fig:mve_error_r} show the errors in the orientation estimates when the Chebyshev center and the maximum volume center are used for position estimation, respectively; the average errors in degrees were $2.68$ and $1.87$, respectively.

\begin{figure}
\centering
\begin{tikzpicture}[scale=0.8]
\begin{axis}[xlabel = {Data point number}, ylabel = {Error in degrees}, xmax = {100}, xmin = {0}]\addplot+ [cyan, mark options={black}, mark size = {0.8}]coordinates {
(1.0, 4.176738567619953)
(2.0, 0.5957042195975168)
(3.0, 3.458033487437137)
(4.0, 3.2598433753968186)
(5.0, 2.225063262741108)
(6.0, 3.0582963036348296)
(7.0, 6.347771910280224)
(8.0, 4.193514771861384)
(9.0, 1.7304574460944246)
(10.0, 1.952802312861868)
(11.0, 1.7233533423925975)
(12.0, 3.3045340834170225)
(13.0, 4.6233643255445855)
(14.0, 0.4190309645955422)
(15.0, 0.08902560924963218)
(16.0, 0.6116181223572755)
(17.0, 0.10596052280031391)
(18.0, 3.565641836856247)
(19.0, 1.788396657211439)
(20.0, 15.334162290249035)
(21.0, 8.91741657467537)
(22.0, 1.7044119306433676)
(23.0, 0.8616632066881132)
(24.0, 0.006936800025648878)
(25.0, 1.9079242476426554)
(26.0, 1.23337014923703)
(27.0, 0.7090971509162678)
(28.0, 3.8366182153588038)
(29.0, 3.5311096705437124)
(30.0, 0.6403055462016807)
(31.0, 6.358960057145743)
(32.0, 0.5332942188050969)
(33.0, 2.257914943840524)
(34.0, 2.505241916391261)
(35.0, 1.2714804369801567)
(36.0, 4.132384758783286)
(37.0, 1.127991751556569)
(38.0, 0.2875858520256044)
(39.0, 4.344937203874148)
(40.0, 0.14006010597752488)
(41.0, 0.4308344681330323)
(42.0, 2.323829727465759)
(43.0, 1.9198835957004214)
(44.0, 4.417604295115)
(45.0, 2.0077346414700354)
(46.0, 4.344944079367689)
(47.0, 8.994325272473576)
(48.0, 2.1395670735094816)
(49.0, 3.30750085888021)
(50.0, 3.204489349851435)
(51.0, 1.3079360226109455)
(52.0, 0.40192015300175526)
(53.0, 2.2950964669977436)
(54.0, 1.2512687777991216)
(55.0, 6.602692865447625)
(56.0, 1.5288278684099366)
(57.0, 0.4170181638612476)
(58.0, 3.628817309262962)
(59.0, 4.350543595899503)
(60.0, 0.9515809112247642)
(61.0, 0.19914981634716178)
(62.0, 2.230215872192719)
(63.0, 10.248547647706596)
(64.0, 0.38560002316524894)
(65.0, 0.7015805176019466)
(66.0, 0.7827588968894671)
(67.0, 1.2197658393494437)
(68.0, 1.7181560422329658)
(69.0, 6.806736314322999)
(70.0, 0.6705072974986167)
(71.0, 3.4493892731819975)
(72.0, 1.9816140685478165)
(73.0, 0.3762522167376895)
(74.0, 0.4082496177645655)
(75.0, 0.41437797434128476)
(76.0, 0.7754021187999872)
(77.0, 7.633642755242886)
(78.0, 5.149274837243832)
(79.0, 2.6884319933550533)
(80.0, 3.2684188347165417)
(81.0, 2.452146490474283)
(82.0, 2.043368605622607)
(83.0, 5.702956613273069)
(84.0, 0.03358220228820781)
(85.0, 3.662098708708726)
(86.0, 3.371979810270848)
(87.0, 3.277246968423918)
(88.0, 3.9336216889478375)
(89.0, 1.513570575283398)
(90.0, 4.009540315676057)
(91.0, 0.6759738878219599)
(92.0, 0.5399004221829553)
(93.0, 3.4714229381515485)
(94.0, 5.807832526967201)
(95.0, 1.478437949201566)
(96.0, 0.09980122650265758)
(97.0, 0.4252137521627989)
(98.0, 1.2848630122010274)
(99.0, 9.3227876524772)
(100.0, 0.33232239666941904)
};
\end{axis}

\end{tikzpicture}
\caption{Geodesic distance in $SO(3)$ between the estimated and actual rotation matrices using Chebyshev center as the position estimate for each on-board receiver.}
\label{fig:cheb_error_r}
\end{figure}
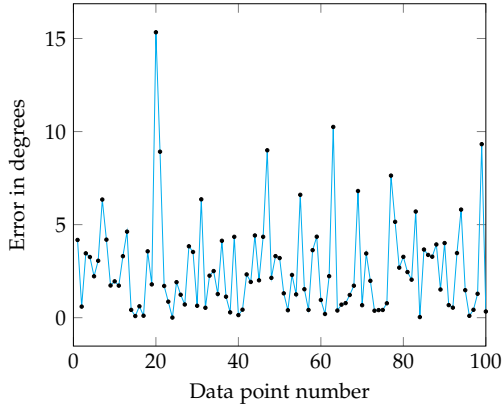

\begin{figure}
\centering
\begin{tikzpicture}[scale=0.8]
\begin{axis}[xlabel = {Data point number}, ylabel = {Error in degrees}, xmax = {100}, xmin = {0}]\addplot+ [cyan, mark options={black}, mark size = {0.8}]coordinates {
(1.0, 4.581304066761827)
(2.0, 0.34625443841582015)
(3.0, 1.462262777687228)
(4.0, 0.5607629614192587)
(5.0, 2.142515514323225)
(6.0, 2.06889444835348)
(7.0, 1.418795334559628)
(8.0, 2.5105131281064645)
(9.0, 0.07575877150337797)
(10.0, 1.9425343362153282)
(11.0, 3.028827938315661)
(12.0, 1.1030663685949926)
(13.0, 3.334538737232434)
(14.0, 1.7645868230769826)
(15.0, 0.18423400606652107)
(16.0, 0.9909396740034809)
(17.0, 0.15435368409265351)
(18.0, 3.266807677396634)
(19.0, 0.6214999254498969)
(20.0, 1.2855333728213305)
(21.0, 1.4576854178619278)
(22.0, 1.6866903460399714)
(23.0, 1.6743792018959953)
(24.0, 0.5897998895186938)
(25.0, 0.8952477008075015)
(26.0, 2.68659795545284)
(27.0, 0.09574124756636056)
(28.0, 2.3288024281696993)
(29.0, 6.464793383315539)
(30.0, 0.9092622484764016)
(31.0, 2.2946902399209956)
(32.0, 3.3316258198019884)
(33.0, 0.2867395933621962)
(34.0, 0.20047449476950424)
(35.0, 1.3060183328706427)
(36.0, 0.8978775770871521)
(37.0, 0.26340531419769825)
(38.0, 0.30075757877786685)
(39.0, 0.2626661986419795)
(40.0, 0.17732528097283362)
(41.0, 0.38348122523885514)
(42.0, 0.6046228806365234)
(43.0, 0.46158110229313765)
(44.0, 0.7336260470836136)
(45.0, 4.753329488129701)
(46.0, 1.7266318068963364)
(47.0, 0.6032099667137306)
(48.0, 0.34734191231097844)
(49.0, 0.8579916931369149)
(50.0, 0.3158298065365783)
(51.0, 0.2601893020936289)
(52.0, 0.04050697020015894)
(53.0, 1.8121589358488994)
(54.0, 1.5035243333035742)
(55.0, 2.2073806392677806)
(56.0, 1.449602129288222)
(57.0, 0.5651008248861942)
(58.0, 0.6748033350465076)
(59.0, 5.210650649216041)
(60.0, 0.2032057845788929)
(61.0, 2.461271416319536)
(62.0, 0.9505925590281634)
(63.0, 6.7831699872514735)
(64.0, 1.1432473258097222)
(65.0, 4.001631779229866)
(66.0, 0.12994396314669507)
(67.0, 0.6160906309060667)
(68.0, 1.6433191598219536)
(69.0, 4.57373873203492)
(70.0, 0.4175086157338796)
(71.0, 0.04566530922972174)
(72.0, 1.8203361895010064)
(73.0, 0.7084605948058774)
(74.0, 0.35028462354677037)
(75.0, 3.0582046303876087)
(76.0, 0.8714619309004406)
(77.0, 10.918178319778663)
(78.0, 2.5371923348789363)
(79.0, 5.172950026296433)
(80.0, 2.1394891512493435)
(81.0, 0.6423487136991172)
(82.0, 0.23931816849039841)
(83.0, 1.9199552154248123)
(84.0, 0.3759221930476942)
(85.0, 1.3152200350604437)
(86.0, 6.782559214241865)
(87.0, 0.4875996887278235)
(88.0, 8.597362350315292)
(89.0, 2.1907688739135525)
(90.0, 0.6167787532180188)
(91.0, 0.17010142909182419)
(92.0, 0.2989957335578396)
(93.0, 3.409117788635637)
(94.0, 9.77148694466241)
(95.0, 1.0612106557450958)
(96.0, 0.01831287704797137)
(97.0, 0.0021142142640327376)
(98.0, 1.1295920226783693)
(99.0, 11.689715265292698)
(100.0, 1.0211133503684555)
};
\end{axis}

\end{tikzpicture}
\caption{Geodesic distance in $SO(3)$ between the estimated and actual rotation matrices using the maximum volume ellipsoid center as the position estimate for each on-board receiver.}
\label{fig:mve_error_r}
\end{figure}
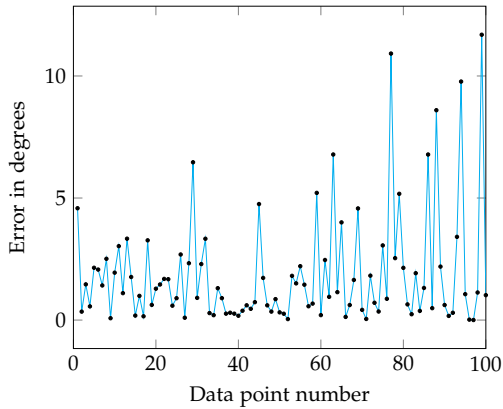

In the next section, we present experimental results that corroborate the algorithms presented in this paper.  

\subsection{Experimental results} 
Experiments to corroborate the effectiveness of the the position estimation algorithms and determination of $\phi(\cdot)$ were conducted at ``The Underwater Systems and Technology Laboratory'' facility in University of Porto. Figure \ref{fig:setup} shows the architecture of the setup at sea. 

\begin{figure}
\centering
\includegraphics[scale=0.2]{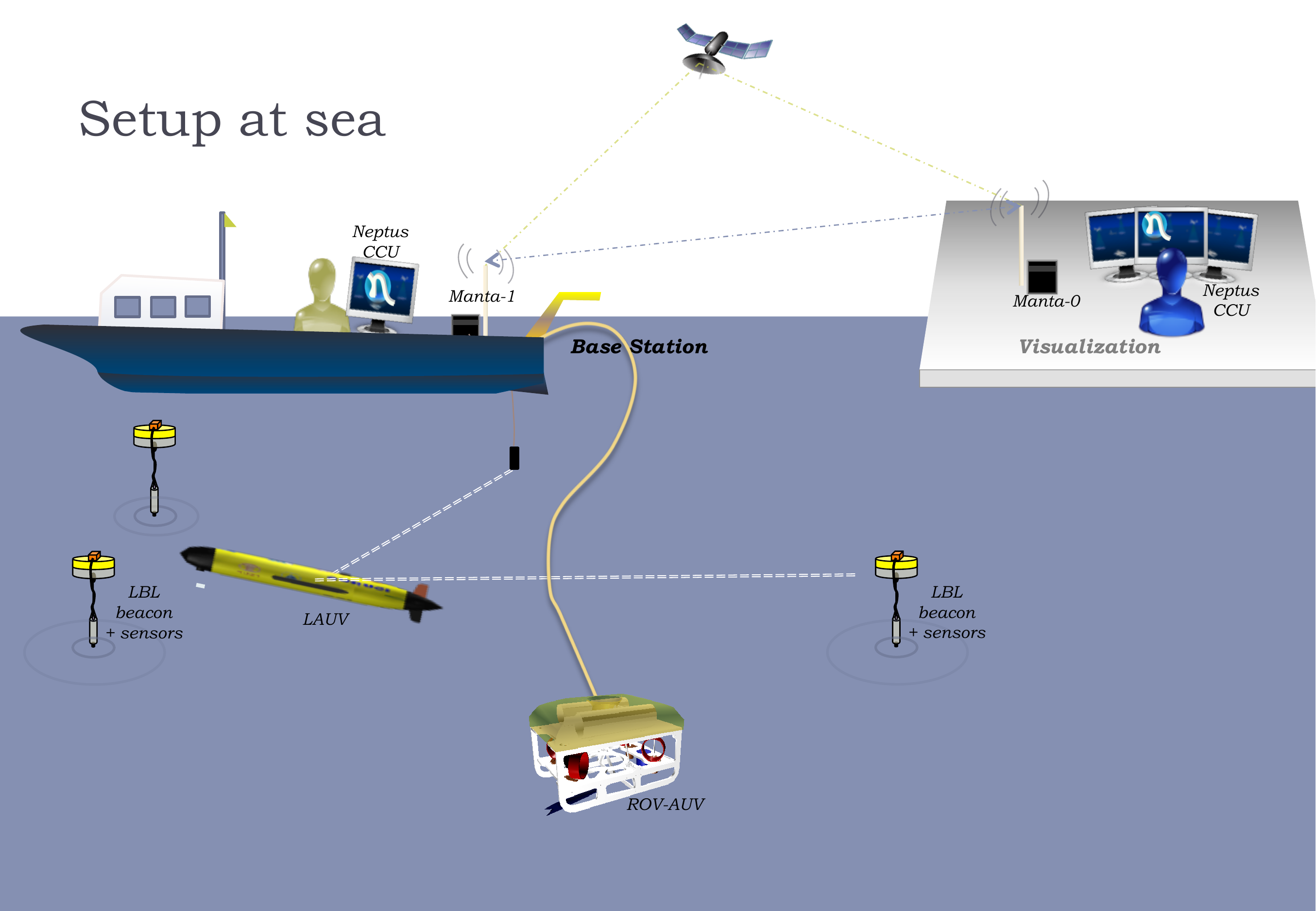}
\caption{Illustration of the experimental setup.}
\label{fig:setup}
\end{figure}

The vehicle is maneuvered through the path shown in green dots in figure \ref{fig:experiments} ten times. For each run, the range measurements are collected at regular intervals of time using the LBL beacons positioned at $4$ locations. Data from nine of these runs are used to compute the function $\phi(\cdot)$. Using this function, the positions are estimated for the last run. The position estimates with the computed function is shown using the blue dots. As it can be observed the position estimates are fairly accurate. The error at few instances can be explained by the fact that the beacon positions are not fixed and they tend to vary with tidal changes and currents driven by wind and thermohaline circulation. Also, the position estimate without prior calibration of the beacons using an arbitrary estimate for the function $\phi(\cdot)$ is shown using red dots; these position estimates are observed to be off by a huge margin. This indicates that the computation of the function $\phi(\cdot)$ by solving a SDP is a crucial step to the position estimation algorithm which in turn also affects the orientation estimates.

\begin{figure}
\centering
\includegraphics[width=7cm, height=5cm]{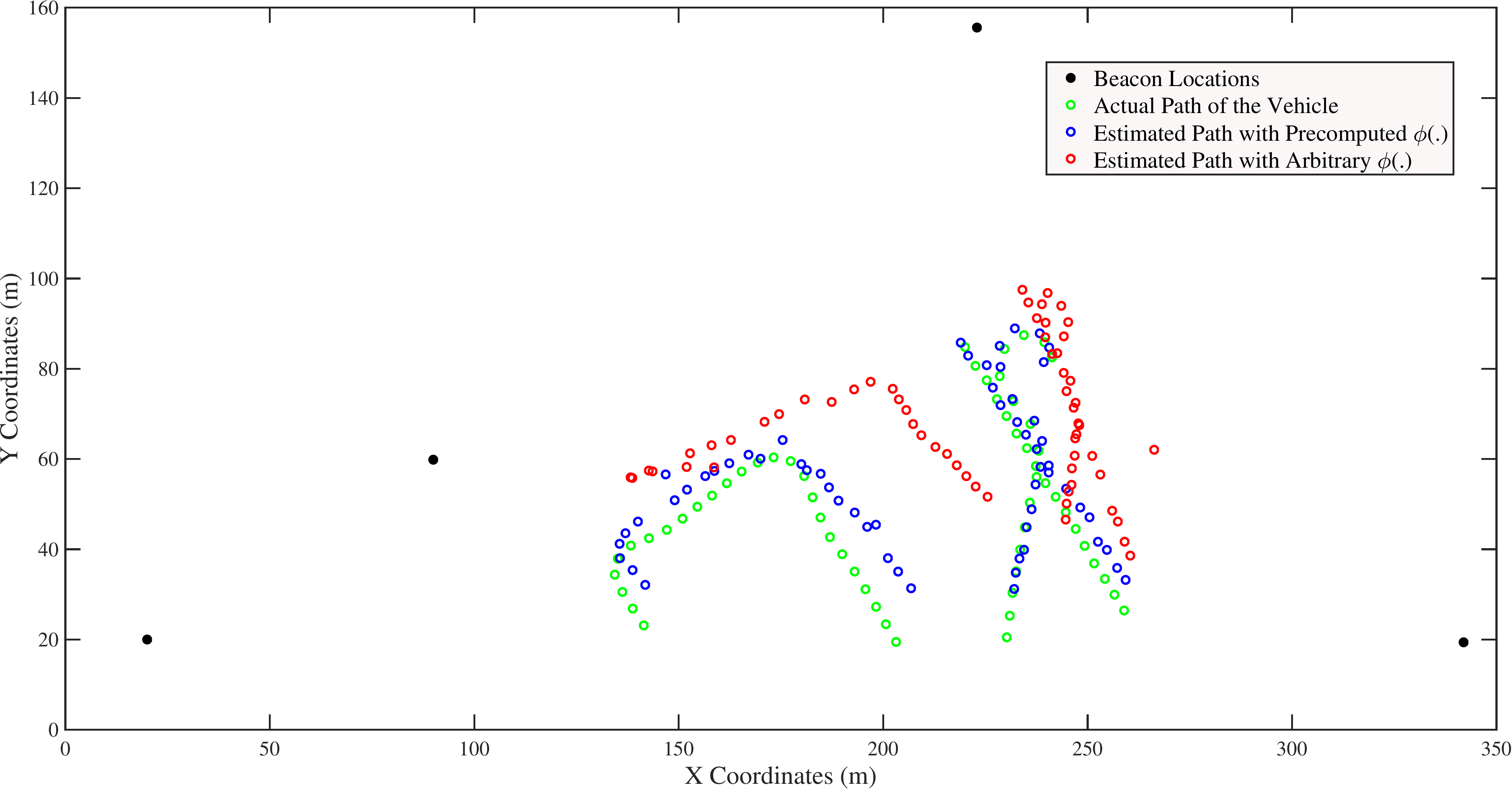}
\caption{Vehicle's actual path and estimated obtained from experiments.}
\label{fig:experiments}
\end{figure}

\section{Conclusion} \label{sec:conclusion}
This article presents algorithms for position and orientation estimation for underwater vehicles using range measurements obtained from LBL acoustic beacons. The position estimates are obtained by solving a convex optimization problem and the algorithm works under the assumption that the vehicle is traveling in the convex hull of the beacons. Simulation and experimental results corroborate the effectiveness of the algorithms. Furthermore, these algorithms are fast and hence can be implemented online. The orientation estimates are obtained by the solution to the ``Orthogonal Procrustes'' problem, obtained via singular value decomposition. A crucial step to the performance of these algorithm is the determination of a sensing function for the sensing model, which can also be computed by solving a convex optimization problem. This computation can also be performed online and in an adaptive fashion as illustrated in the simulation results. Future work can be focused on relaxing the assumption of the vehicle moving in the convex hull of the acoustic beacons and accounting for uncertainty in the beacon locations. 

\section*{Acknowledgements}
Kaarthik Sundar acknowledges the support of the U.S. Department of Energy through the LANL/LDRD Program and the Center for Nonlinear Studies for this work.









\appendix
\section*{Appendix A: Proof of Proposition \ref{prop:mve}}
\noindent For a fixed value of $i\in I$, Eq. \eqref{eq:mve-spheres} is given by 
\begin{flalign*}
\norm{P_j\bm u + \bm r_{c,j} - \bm B_i} \leqslant \phi(D_{ij}), \quad\forall \{\bm u:\norm{\bm u}^2 \leqslant 1\}.
\end{flalign*}
We can convert the above constraint to an equivalent semi-definite constraint using Schur Complement as follows:
\begin{flalign*}
\begin{bmatrix}
\phi(D_{ij}) & (P_j \bm u + \bm r_{c,j} - \bm B_i)^T \\ 
P_j \bm u + \bm r_{c,j} - \bm B_i & \phi(D_{ij})I_3
\end{bmatrix} \succeq 0, \\
\qquad \qquad \qquad \qquad \forall \{\bm u:\norm{\bm u} \leqslant 1\}. 
\end{flalign*}
Equivalently, 
\begin{flalign*}
& \Leftrightarrow x^2\phi(D_{ij}) + 2x\bm y^T(P_j \bm u + \bm r_{c,j} - \bm B_i) + \phi(D_{ij}) \bm y^T \bm y \geqslant 0 & \\
& \qquad \qquad \qquad \qquad \forall [x;\bm y], ~  \{\bm u:\norm{\bm u} \leqslant 1\}, & 
\end{flalign*}
\begin{flalign*}
& \Leftrightarrow x^2\phi(D_{ij}) + \min_{\bm u:\norm{\bm u} \leqslant 1} 2x\bm y^T P_j \bm u +2x\bm y^T(\bm r_{c,j} - \bm B_i) + & \\
& \qquad \qquad \phi(D_{ij}) \bm y^T \bm y \geqslant 0, \quad \forall [x;\bm y], &
\end{flalign*}
\begin{flalign*}
& \Leftrightarrow x^2\phi(D_{ij}) - 2x \norm{P_j \bm y} +2x\bm y^T(\bm r_{c,j} - \bm B_i) + & \\
& \qquad \qquad \phi(D_{ij}) \bm y^T \bm y \geqslant 0, \quad \forall [x;\bm y], &
\end{flalign*}
\begin{flalign*}
& \Leftrightarrow x^2\phi(D_{ij}) + 2 \bm y^T P_j \bm{\phi} +2x\bm y^T(\bm r_{c,j} - \bm B_i) + & \\
& \qquad \qquad  + \phi(D_{ij}) \bm y^T \bm y \geqslant 0, \quad \forall \{(x,\bm y,\bm{\phi}): \bm{\phi}^T\bm{\phi} \leq x^2\},  &
\end{flalign*}
\begin{flalign*}
& \Leftrightarrow \exists \lambda \geq 0 :
\begin{bmatrix}
\phi(D_{ij})-\lambda & (\bm r_{c,j}-\bm B_i)^T & 0 \\ 
(\bm r_{c,j} - \bm B_i) & \phi(D_{ij}) I_3 & P_j \\ 
0 & P_j & \lambda I_3
\end{bmatrix} \succeq 0.&
\end{flalign*}
The last two equivalences follow from Cauchy-Schwarz inequality and the $\mathcal S$-lemma \cite{Boyd2004}, respectively. \hfill \qed

\bibliographystyle{unsrt}
\bibliography{Localization.bib}

\begin{thebibliography}{10}

\bibitem{Leonard1998}
John~J Leonard, Andrew~A Bennett, Christopher~M Smith, and H~Feder.
\newblock Autonomous underwater vehicle navigation.
\newblock In {\em IEEE ICRA Workshop on Navigation of Outdoor Autonomous
  Vehicles, Leuven, Belgium, May}. Citeseer, 1998.

\bibitem{Paull2014}
Liam Paull, Sajad Saeedi, Mae Seto, and Howard Li.
\newblock Auv navigation and localization: A review.
\newblock {\em IEEE Journal of Oceanic Engineering}, 39(1):131--149, 2014.

\bibitem{Sundar2017}
Kaarthik Sundar, Sohum Misra, Sivakumar Rathinam, and Rajnikant Sharma.
\newblock Routing unmanned vehicles in gps-denied environments.
\newblock In {\em Unmanned Aircraft Systems (ICUAS), 2017 International
  Conference on}, pages 62--71. IEEE, 2017.

\bibitem{Sundar2018}
Kaarthik Sundar, Shriram Srinivasan, Sohum Misra, Sivakumar Rathinam, and
  Rajnikant Sharma.
\newblock Landmark placement for localization in a {GPS}-denied environment.
\newblock {\em arXiv preprint arXiv:1802.07652}, 2018.

\bibitem{Zhang2015}
Xi~Zhang, Jun-Hong Cui, Santanu Das, Mario Gerla, and Mandar Chitre.
\newblock Underwater wireless communications and networks: theory and
  application: Part 1 [guest editorial].
\newblock {\em IEEE Communications Magazine}, 53(11):40--41, 2015.

\bibitem{Armstrong2010}
Benjamin Armstrong, Eric Wolbrecht, and DB~Edwards.
\newblock {AUV} navigation in the presence of a magnetic disturbance with an
  extended {Kalman} filter.
\newblock In {\em OCEANS 2010 IEEE-Sydney}, pages 1--6. IEEE, 2010.

\bibitem{Armstrong2009}
Benjamin Armstrong, Jesse Pentzer, Douglas Odell, Thomas Bean, John Canning,
  Donald Pugsley, James Frenzel, Michael Anderson, and Dean Edwards.
\newblock Field measurement of surface ship magnetic signature using multiple
  {AUVs}.
\newblock In {\em OCEANS 2009, MTS/IEEE Biloxi-Marine Technology for Our
  Future: Global and Local Challenges}, pages 1--9. IEEE, 2009.

\bibitem{Tang2017}
Jingtian Tang, Shuanggui Hu, Zhengyong Ren, Chaojian Chen, Xiao Xiao, and Cong
  Zhou.
\newblock Analytical formulas for underwater and aerial object localization by
  gravitational field and gravitational gradient tensor.
\newblock {\em IEEE Geoscience and Remote Sensing Letters}, 14(9):1557--1560,
  2017.

\bibitem{Bishop2002}
Garner~C Bishop.
\newblock Gravitational field maps and navigational errors [unmanned underwater
  vehicles].
\newblock {\em IEEE Journal of Oceanic Engineering}, 27(3):726--737, 2002.

\bibitem{Qureshi2016}
Umair~Mujtaba Qureshi, Faisal~Karim Shaikh, Zuneera Aziz, Syed M Zafi~S Shah,
  Adil~A Sheikh, Emad Felemban, and Saad~Bin Qaisar.
\newblock {RF} path and absorption loss estimation for underwater wireless
  sensor networks in different water environments.
\newblock {\em Sensors}, 16(6):890, 2016.

\bibitem{Niculescu2003}
Drago{\c{s}} Niculescu and Badri Nath.
\newblock {DV} based positioning in ad hoc networks.
\newblock {\em Telecommunication Systems}, 22(1-4):267--280, 2003.

\bibitem{Wong2005}
Sau~Yee Wong, Joo~Ghee Lim, SV~Rao, and Winston~KG Seah.
\newblock Multihop localization with density and path length awareness in
  non-uniform wireless sensor networks.
\newblock In {\em Vehicular Technology Conference, 2005. VTC 2005-Spring. 2005
  IEEE 61st}, volume~4, pages 2551--2555. IEEE, 2005.

\bibitem{Bulusu2000}
Nirupama Bulusu, John Heidemann, and Deborah Estrin.
\newblock {GPS}-less low-cost outdoor localization for very small devices.
\newblock {\em IEEE personal communications}, 7(5):28--34, 2000.

\bibitem{Chandrasekhar2006}
Vijay Chandrasekhar, Winston~KG Seah, Yoo~Sang Choo, and How~Voon Ee.
\newblock Localization in underwater sensor networks: survey and challenges.
\newblock In {\em Proceedings of the 1st ACM international workshop on
  Underwater networks}, pages 33--40. ACM, 2006.

\bibitem{Eustice2011}
Ryan~M Eustice, Hanumant Singh, and Louis~L Whitcomb.
\newblock Synchronous-clock, one-way-travel-time acoustic navigation for
  underwater vehicles.
\newblock {\em journal of field robotics}, 28(1):121--136, 2011.

\bibitem{Deffenbaugh1994}
Max Deffenbaugh.
\newblock {\em A matched field processing approach to long range acoustic
  navigation}.
\newblock PhD thesis, Massachusetts Institute of Technology and Woods Hole
  Oceanographic Institution, 1994.

\bibitem{Deffenbaugh1996}
Max Deffenbaugh, Henrik Schmidt, and James~G Bellingham.
\newblock Acoustic positioning in a fading multipath environment.
\newblock In {\em OCEANS'96. MTS/IEEE. Prospects for the 21st Century.
  Conference Proceedings}, volume~2, pages 596--600. IEEE, 1996.

\bibitem{Ash2005}
Joshua~N Ash and Randolph~L Moses.
\newblock Acoustic time delay estimation and sensor network self-localization:
  Experimental results.
\newblock {\em The Journal of the Acoustical Society of America},
  118(2):841--850, 2005.

\bibitem{Hahn2005}
Matthew~J Hahn.
\newblock {\em Undersea navigation via a distributed acoustic communications
  network}.
\newblock PhD thesis, Monterey California. Naval Postgraduate School, 2005.

\bibitem{Garcia2005}
J~Esteban Garcia.
\newblock Positioning of sensors in underwater acoustic networks.
\newblock In {\em OCEANS, 2005. Proceedings of MTS/IEEE}, pages 2088--2092.
  IEEE, 2005.

\bibitem{Whitcomb1999}
Louis Whitcomb, Dana Yoerger, and Hanumant Singh.
\newblock Advances in doppler-based navigation of underwater robotic vehicles.
\newblock In {\em Robotics and Automation, 1999. Proceedings. 1999 IEEE
  International Conference on}, volume~1, pages 399--406. IEEE, 1999.

\bibitem{Kinsey2004}
James~C Kinsey and Louis~L Whitcomb.
\newblock Preliminary field experience with the dvlnav integrated navigation
  system for oceanographic submersibles.
\newblock {\em Control Engineering Practice}, 12(12):1541--1549, 2004.

\bibitem{Kinsey2006}
James~C Kinsey, Ryan~M Eustice, and Louis~L Whitcomb.
\newblock A survey of underwater vehicle navigation: Recent advances and new
  challenges.
\newblock In {\em IFAC Conference of Manoeuvering and Control of Marine Craft},
  volume~88, pages 1--12, 2006.

\bibitem{Hari2015}
Sai Krishna~Kanth Hari and Swaroop Darbha.
\newblock Estimation of location and orientation from range measurements.
\newblock In {\em {ASME} 2015 {Dynamic} {Systems} and {Control} {Conference}},
  pages V002T31A005--V002T31A005. American Society of Mechanical Engineers,
  2015.

\bibitem{Hari2017}
Sai Krishna~Kanth Hari, Kaarthik Sundar, Jos{\'e} Braga, Jo{\~a}o Teixeira,
  Swaroop Darbha, and Jo{\~a}o Sousa.
\newblock Adaptive position estimation for vehicles using range measurements.
\newblock {\em IFAC-PapersOnLine}, 50(1):1145--1150, 2017.

\bibitem{Boyd2004}
Stephen Boyd and Lieven Vandenberghe.
\newblock {\em Convex optimization}.
\newblock Cambridge university press, 2004.

\bibitem{Todd2016}
Michael~J Todd.
\newblock {\em Minimum-Volume Ellipsoids: Theory and Algorithms}, volume~23.
\newblock SIAM, 2016.

\bibitem{Polya1976}
G~P{\'o}lya and G~Szeg\"{o}.
\newblock {\em Problems and Theorems in Analysis II}.
\newblock Springer-Verlag, New York, 1976.

\bibitem{Parrilo2003}
Pablo~A Parrilo.
\newblock Semidefinite programming relaxations for semialgebraic problems.
\newblock {\em Mathematical programming}, 96(2):293--320, 2003.

\bibitem{SCS2013}
Brendan O'Donoghue, Eric Chu, Neal Parikh, and Stephen Boyd.
\newblock Conic optimization via operator splitting and homogeneous self-dual
  embedding.
\newblock {\em arXiv preprint arXiv}, 1312, 2013.

\bibitem{Schoenemann1966}
Peter~H. Sch\"{o}nemann.
\newblock A generalized solution of the orthogonal {Procrustes} problem.
\newblock {\em Psychometrika}, 31(1):1--10, 1966.

\bibitem{Huynh2009}
Du~Q Huynh.
\newblock Metrics for {3D} rotations: Comparison and analysis.
\newblock {\em Journal of Mathematical Imaging and Vision}, 35(2):155--164,
  2009.

\end{thebibliography}

\end{document}